\pdfoutput=1  %for arXiv rendering of xymatrix to come out right

\documentclass[11pt]{article}
\usepackage[margin=.8in,left=.8in]{geometry}
\usepackage{amsthm}
\usepackage{amssymb}
\usepackage{amsmath}
\usepackage{mathtools}
\usepackage{xcolor}
\usepackage{stmaryrd}    % for \sslash
\usepackage{hyperref}
\usepackage[all]{xy}

\usepackage{adjustbox}
\usepackage{hyperref}
\usepackage{multirow}
\usepackage{xfrac}

\usepackage{tikz}
\usetikzlibrary{cd}

\usepackage{rotating}

\DeclareMathAlphabet{\mathpzc}{OT1}{pzc}{m}{it} %for math script

\newcommand\mathscr[1]{\scalebox{1.1}{$\mathpzc{#1}$}}

\usepackage[titletoc]{appendix}

\newcommand{\underoverset}[3]{\underset{#1}{\overset{#2}{#3}}}
\newcommand{\gt}{>}

\newcommand{\boldpi}{\mbox{$\pi$\hspace{-6.3pt}$\pi$}}

\definecolor{darkblue}{rgb}{0.05,0.25,0.65}
\definecolor{greenii}{RGB}{20,140,10}
\definecolor{lightgray}{rgb}{0.9,0.9,0.9}
\definecolor{orangeii}{RGB}{200,100,5}

\usepackage{multirow}

\usepackage{enumerate} %Roman enumerate

\usepackage{mathptmx}
\usepackage{amsmath}
\usepackage{graphicx}

\usepackage[new]{old-arrows}   %For \longhookrightarrow

\newdir{> }{{}*!/10pt/@{>}}

\usepackage{amssymb}

\makeatletter
\newif\if@sup
\newtoks\@sups
\def\append@sup#1{\edef\act{\noexpand\@sups={\the\@sups #1}}\act}%
\def\reset@sup{\@supfalse\@sups={}}%
\def\mk@scripts#1#2{\if #2/ \if@sup ^{\the\@sups}\fi \else%
  \ifx #1_ \if@sup ^{\the\@sups}\reset@sup \fi {}_{#2}%
  \else \append@sup#2 \@suptrue \fi%
  \expandafter\mk@scripts\fi}
\def\tensor#1#2{\reset@sup#1\mk@scripts#2_/}
\def\multiscripts#1#2#3{\reset@sup{}\mk@scripts#1_/#2%
  \reset@sup\mk@scripts#3_/}
\makeatother

\makeatletter
\newbox\slashbox \setbox\slashbox=\hbox{$/$}
\def\itex@pslash#1{\setbox\@tempboxa=\hbox{$#1$}
  \@tempdima=0.5\wd\slashbox \advance\@tempdima 0.5\wd\@tempboxa
  \copy\slashbox \kern-\@tempdima \box\@tempboxa}
\def\slash{\protect\itex@pslash}
\makeatother

\def\clap#1{\hbox to 0pt{\hss#1\hss}}
\def\mathllap{\mathpalette\mathllapinternal}
\def\mathrlap{\mathpalette\mathrlapinternal}
\def\mathclap{\mathpalette\mathclapinternal}
\def\mathllapinternal#1#2{\llap{$\mathsurround=0pt#1{#2}$}}
\def\mathrlapinternal#1#2{\rlap{$\mathsurround=0pt#1{#2}$}}
\def\mathclapinternal#1#2{\clap{$\mathsurround=0pt#1{#2}$}}

\let\oldroot\root
\def\root#1#2{\oldroot #1 \of{#2}}
\renewcommand{\sqrt}[2][]{\oldroot #1 \of{#2}}

\DeclareSymbolFont{symbolsC}{U}{txsyc}{m}{n}
\SetSymbolFont{symbolsC}{bold}{U}{txsyc}{bx}{n}
\DeclareFontSubstitution{U}{txsyc}{m}{n}

\DeclareSymbolFont{stmry}{U}{stmry}{m}{n}
\SetSymbolFont{stmry}{bold}{U}{stmry}{b}{n}

\DeclareFontFamily{OMX}{MnSymbolE}{}
\DeclareSymbolFont{mnomx}{OMX}{MnSymbolE}{m}{n}
\SetSymbolFont{mnomx}{bold}{OMX}{MnSymbolE}{b}{n}
\DeclareFontShape{OMX}{MnSymbolE}{m}{n}{
    <-6>  MnSymbolE5
   <6-7>  MnSymbolE6
   <7-8>  MnSymbolE7
   <8-9>  MnSymbolE8
   <9-10> MnSymbolE9
  <10-12> MnSymbolE10
  <12->   MnSymbolE12}{}

\usepackage{cleveref}

\crefformat{section}{\S#2#1#3} % see manual of cleveref, section 8.2.1
\crefformat{subsection}{\S#2#1#3}
\crefformat{subsubsection}{\S#2#1#3}

\theoremstyle{italics}
\newtheorem{theorem}{Theorem}[section]
\newtheorem{lemma}[theorem]{Lemma}
\newtheorem{prop}[theorem]{Proposition}
\newtheorem{cor}[theorem]{Corollary}

\newtheorem{question}[theorem]{Question}
\theoremstyle{definition}
\newtheorem{defn}[theorem]{Definition}
\newtheorem{notation}[theorem]{Notation}
\newtheorem{example}[theorem]{Example}
\newtheorem{remark}[theorem]{Remark}

\usepackage{amsfonts}
\usepackage{colortbl}

\renewcommand{\emph}{\textit}

\begin{document}

%%%%%%%%%%%%%%%%%%%%%%%%%%%%%%%%%%%%%%%%%%%%%%%%%%
\title{
  Fundamental weight systems are quantum states
  %Weight systems that are quantum states
}
%%%%%%%%%%%%%%%%%%%%%%%%%%%%%%%%%%%%%%%%%%%%%%%%%%

\author{
  David Corfield, \quad Hisham Sati, \quad Urs Schreiber
}

\maketitle

\begin{abstract}
  Weight systems on chord diagrams play a central role
  in knot theory and Chern-Simons theory; and
  more recently in stringy quantum gravity.
  We highlight that the noncommutative algebra of horizontal chord diagrams is
  canonically a star-algebra, and ask which weight systems are
  positive with respect to this structure; hence we ask:
  Which weight systems are quantum states, if horizontal chord diagrams are
  quantum observables?
  We observe that the fundamental $\mathfrak{gl}(n)$-weight systems
  on horizontal chord diagrams with $N$ strands
  may be identified with the Cayley distance kernel
  at inverse temperature $\beta = \mathrm{ln}(n)$
  on the symmetric group on $N$ elements. In contrast to
  related kernels like the Mallows kernel, the positivity of the
  Cayley distance kernel had remained open.
  We characterize its phases of indefinite, semi-definite and definite
  positivity, in dependence of the inverse temperature $\beta$; and we prove
  that the Cayley distance kernel is positive (semi-)definite
  at $\beta = \mathrm{ln}(n)$ for all $n = 1,2,3, \cdots$.
  In particular, this proves that
  all fundamental $\mathfrak{gl}(n)$-weight systems
  are quantum states, and hence so are all their convex combinations.
  We close with briefly recalling
  how,
  under our
  ``Hypothesis H'',
  this result impacts on the identification of bound states of multiple M5-branes.
\end{abstract}

\medskip

\tableofcontents

\vfill

\newpage

%%%%%%%%%%%%%%%%%%%%%%%%%%%%%%%%%%%%%%%%%%%%%%%
\section{Introduction}
%%%%%%%%%%%%%%%%%%%%%%%%%%%%%%%%%%%%%%%%%%%%%%%

In investigations of problems in
string/M-theory (\cite{SS19b}, surveyed below in \cref{PhysicsInterpretation}),
we encountered the following
question at the interface of
quantum topology and
quantum probability theory:

\begin{question}
  \label{TheQuestion}
  Which weight systems on horizontal chord diagrams
  (Def. \ref{AlgebraOfAndWeightSystemsOnHorizontalChordDiagrams})
  are quantum states (Def. \ref{NotionOfWeightSystemsThatAreQuantumStates})
  with respect to the star-operation of reversal of strands
  (Prop. \ref{StarStructureOnHorizontalChordDiagrams})?
\end{question}

It is known
that all weight systems are {\it generated},
in a sense (\cite[Cor. 2.6]{BarNatan96}, review in \cite[\S 3.4]{SS19b}),
from
{\it Lie algebra weight systems}
$w_{(\mathfrak{g},\mathbf{V})}$
induced by a metric Lie module $(\mathfrak{g}, \mathbf{V})$
(\cite[\S 2.4]{BarNatan95}, review in \cite[\S 6]{CDM11}\cite[\S 3.3]{SS19b}),
and in fact from just those with $\mathfrak{g} := \mathfrak{gl}(n)$ and
$\mathbf{V} := \mathbf{n}$ the fundamental representation
(in the sense of \cite[Fact 7]{BarNatan96}).
Here we prove that all these {\it fundamental weight systems}
(Def. \ref{FundamentalWeightSystemOfgl2})
are quantum states:

\begin{theorem}%[Fundamental weight systems are quantum states]
  \label{FundamentalglNWeightSystemsAreQuantumStates}
The fundamental $\mathfrak{gl}(n)$-weight systems $w_{(\mathfrak{gl}(n), \mathbf{n})}$
for $n \in \mathbb{N}_+$
are quantum states on horizontal chord diagrams on $N$ strands,
for all $N \in \mathbb{N}_+$; hence so are the
mixtures (Ex. \ref{Mixture}) of their operator images (Ex. \ref{OperatorStateCorrespondence}).

\end{theorem}
\noindent This turns out to be a consequence of the following more general theorem in
geometric group theory:

\medskip

It is well-known that the value of a fundamental Lie algebra weight system
$w_{(\mathfrak{gl}(n), \mathbf{n})}$ depends only on the
permutation induced \eqref{MonoidHomomorphismFromHorizontalChordsToPermutations} by
a horizontal chord diagram
(\cite[Prop. 2.1]{BarNatan96}, review in \cref{WeightSystemsOnChordDiagrams} below).
More concretely, we highlight (Prop. \ref{FundamentalWeightSystemExpressedViaCayleyGraphDistance})
that
%,
%when weight systems are equivalently regarded as sesquilinear forms %\eqref{InducedBilinearFormFromLinearFormOnStarAlgebra},
this is the special value
at inverse temperature $\beta = \mathrm{ln}(n)$
of the
{\it Cayley distance kernel} (Def. \ref{CayleyDistanceKernel})
on permutations of $N$ elements:
\vspace{-1mm}
$$
  \begin{tikzcd}
    \underset{
      \mathclap{
      \raisebox{-6pt}{
        \tiny
        \color{darkblue}
        \bf
        \begin{tabular}{c}
          fundamental
          \\
          weight system
        \end{tabular}
      }
      }
    }{
      w_{(\mathfrak{gl}(n), \mathbf{n})}
    }
    \quad
    \ar[
      rr,
      <->,
      "{
        \mbox{\tiny Prop. \ref{FundamentalWeightSystemExpressedViaCayleyGraphDistance}}
      }"
    ]
    &&
    \quad
    \underset{
      \mathclap{
      \raisebox{-3pt}{
        \tiny
        \color{darkblue}
        \bf
        \begin{tabular}{c}
          Cayley distance kernel at
          \\
          log-integral inverse temperature
        \end{tabular}
      }
      }
    }{
      \big[
        e^{- \beta \cdot d_C}
      \big]
      \;
      \scalebox{.8}{for $\beta = \mathrm{ln}(n)$}
    }
  \end{tikzcd}
$$

\vspace{-2mm}
\noindent
Positive (semi-)definite {\it kernels} are of interest notably
in geometric group theory (e.g. \cite[\S 2.11]{DrutuKapovich18})
and in machine learning (e.g. \cite{HSS07}\cite{MairalVert17}).
While related kernels have recently been proven to be
positive-definite \cite{JiaoVert18}, the archetypical Cayley distance kernel
was known to become indefinite at sufficiently low inverse temperature $\beta$,
while its general behavior with $\beta$ had remained unknown (Rem. \ref{ReferencesOnCayleyDistanceKernel}).
Here we prove:
\begin{theorem}[Phases of the Cayley distance kernel]
\label{PhasesOfTheCayleyDistanceKernel}
The Cayley distance kernel $e^{- \beta \cdot d_C}$
on the symmetric group on $N$ elements is:
$$
  \left\{
  \def\arraystretch{1.4}
  \begin{array}{lclc}
    \mbox{indefinite}
      & \mbox{for} &
    e^{\beta} \in [1,N-1]
      \setminus
    \mathrlap{
      \{1, 2, \cdots, N-1\}
    }
    \\
    \mbox{positive semi-definite}
      &\mbox{for}&
    e^\beta \in \{1, 2, \cdots, N-1\}
    &
    \multirow{2}{*}{$
      \left.
        \mathclap{
        \begin{array}{c}
          {\phantom{A}}
          \\
          {\phantom{A}}
        \end{array}
        }
      \right\}
      \;\;\Rightarrow\!\!\!
      \mbox{
        \def\arraystretch{.8}
        {\begin{tabular}{l}
          fundamental weight systems
          \\
          $w_{(\mathfrak{gl}(n), \mathbf{n})}$ are quantum states
        \end{tabular}}
      }
    $}
    \\
    \multirow{2}{*}{
      \mbox{positive definite}
    }
    &
    \multirow{2}{*}{
      \mbox{for}
      $
      \mathrlap{
      \left\{
         \begin{array}{c}
           \mathclap{\phantom{A}}
           \\
           \mathclap{\phantom{A}}
         \end{array}
      \right.
      }
      $
    }
    & e^\beta \in \{N, N+1, N+2, \cdots\}
    \\
    &&
    e^{\beta} \gt N - 1
    \,.
  \end{array}
  \right.
$$
\end{theorem}
\begin{proof}
  This is the content of
  Prop. \ref{CayleyDistanceKernelIndefiniteAtNonIntegerExpInverseTemperatureBelowNMinusOne},
  Prop. \ref{CayleyDistanceKernelPositiveSemiDefiniteAtLowIntegers},
  and
  Prop. \ref{CayleyDistanceKernelPositiveDefiniteForSufficientlyLargeBeta} below.
\end{proof}

\hypertarget{Figure1}{}
\begin{tabular}{cc}
\hspace{-.9cm}
\begin{minipage}[left]{7.2cm}
For illustration, the blue graph in
\hyperlink{Figure1}{Figure 1}
shows, vertically, the value of the
smallest eigenvalue
(rescaled by $e^{3 \beta}$, for visibility)
of the Cayley distance kernel on the symmetric group
$\mathrm{Sym}(4)$,
in dependence of the exponentiated
inverse temperature $e^\beta$ (running horizontally).
This means: where the graph is negative/zero/positive, the Cayley distance kernel
is indefinite/positive semi-definite/positive definite, respectively.
See also \cite{BarNatan21} for
more such computer algebra analysis of the situation.
\end{minipage}
&
\hspace{-.3cm}
\raisebox{0pt}{
\begin{tabular}{l}
\begin{tikzpicture}[scale=(1.2)]
\draw (0,0) node {
  \includegraphics[width=.48\textwidth]{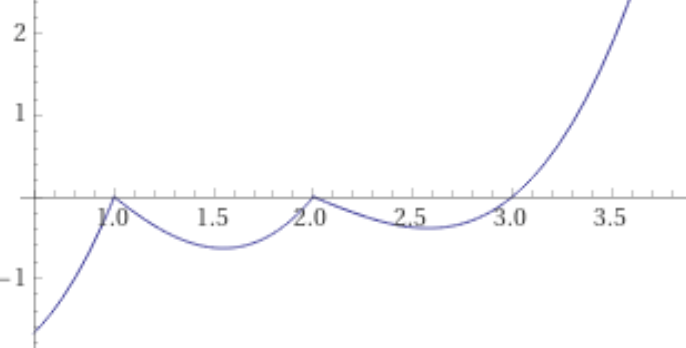}
};
\draw (-3,1.6) node
  {\scalebox{.8}{
  $\mathrlap{
     \mathrm{min}\big(\mathrm{EigVals}[e^{- \beta \cdot d_C}]\big)
     \cdot e^{3\beta}
   }
  $}};
\draw (3.8,-.24) node {\scalebox{.8}{
  $
    e^\beta
  $}};
\draw (-.315,.08) to (-.315,-.17);
\node at (-.315, .3) [rectangle,draw] {\scalebox{.6}{
  $\!\!\!\!w_{(\mathfrak{gl}(2), \mathbf{2} )}\!\!\!\!$
}};

\begin{scope}[shift={(-2.02,0)}]
\draw (-.315,.08) to (-.315,-.17);
\node at (-.315, .3) [rectangle,draw] {\scalebox{.6}{
  $\!\!\!\!w_{(\mathfrak{gl}(1), \mathbf{1} )}\!\!\!\!$
}};
\end{scope}

\begin{scope}[shift={(2.05,0)}]
\draw (-.315,.08) to (-.315,-.17);
\node at (-.315-.2, .3) [rectangle,draw] {\scalebox{.6}{
  $\!\!\!\!w_{(\mathfrak{gl}(3), \mathbf{3} )}\!\!\!\!$
}};
\end{scope}
\end{tikzpicture}
\\
\footnotesize
{\bf Figure 1.} Smallest eigenvalue of the Cayley distance kernel
on $\mathrm{Sym}(4)$.
\end{tabular}
}
\end{tabular}

\bigskip

\noindent {\bf Acknowledgements.}
We thank
Abdelmalek Abdesselam,
Dror Bar-Natan,
Carlo Collari,
Jean-Philippe Vert
and
David Speyer
for comments and discussion.

\medskip

\newpage

%%%%%%%%%%%%%%%%%%%%%%%%%%%%%%%%%%%%%%%%%
\section{Weights, states and kernels}
\label{Definitions}
%%%%%%%%%%%%%%%%%%%%%%%%%%%%%%%%%%%%%%%%%

We briefly recall relevant definitions and facts,
and then point out some close relations between the following topics:

- \cref{WeightSystemsOnChordDiagrams}  -- Weight systems on chord diagrams.

- \cref{QuantumStatesOnQuantumObservableAlgebras} -- Quantum states on quantum observable algebras.

- \cref{GeometricGroupTheory} -- Cayley distance kernels on symmetric groups.

\medskip

\noindent
\begin{notation}
  \label{TheParameters}
  Throughout we consider the following parameters:

  \vspace{-.1cm}
  \begin{enumerate}[{\bf (i)}]
  \vspace{-.2cm}
  \item
  $N \in \mathbb{N}_+$
  the number of strands in horizontal chord diagrams (Def. \ref{AlgebraOfAndWeightSystemsOnHorizontalChordDiagrams}), equivalently:
  \\
  \phantom{$N \in \mathbb{N}_+$}
  the number of elements on which permutations act;

  \vspace{-.2cm}
  \item
  $n \in \mathbb{N}_+\;$
  labels
  the fundamental weight systems $w_{(\mathfrak{gl}(n), \mathbf{n})}$ (Def. \ref{FundamentalWeightSystemOfgl2});

  \vspace{-.2cm}
  \item
  $\beta \in \mathbb{R}_{\geq 0}$
  an inverse temperature parameter (Def. \ref{CayleyDistanceKernel}),
  \\
  \phantom{$\beta \in \mathbb{R}_{\geq 0}$}
  often specialized to $\beta  = \mathrm{ln}(n)$ (Prop. \ref{FundamentalWeightSystemExpressedViaCayleyGraphDistance}).
\end{enumerate}
\end{notation}
The ground field is the complex numbers $\mathbb{C}$, in which
the complex conjugate of $z$ is denoted $\bar z$.

%%%%%%%%%%%%%%%%%%%%%%%%%%%%%%%%%%%%%%%%%%%%%%%%%
\subsection{Weight systems on chord diagrams}
\label{WeightSystemsOnChordDiagrams}
%%%%%%%%%%%%%%%%%%%%%%%%%%%%%%%%%%%%%%%%%%%%%%%%%

\begin{defn}[Horizontal chord diagrams and weight systems
({\cite{BarNatan96}\cite[\S 5.11]{CDM11}\cite[\S 3.1]{SS19b}})]
\label{AlgebraOfAndWeightSystemsOnHorizontalChordDiagrams}
$\,$

\noindent
{\bf (i)}
The {\it monoid of horizontal chord diagrams} is the free monoid
on the set of pairs of distinct strands
\begin{equation}
  \label{SetOfHorizontalChordDiagrams}
  \mathcal{D}^{{}^{\mathrm{pb}}}_{\!N}
  \;\coloneqq\;
  \mathrm{FreeMonoid}
  \Big(
    \big\{
      (ij)
      \,\vert\,
      1 \leq i < j \leq N
    \big\}
  \Big)
  \,,
\end{equation}
where the generator $(i j)$ is called the
{\it chord connecting the $i$th and $j$th strand}.
Hence a general horizontal chord diagram is a finite list of chords
\begin{equation}
  \label{HorizontalChordDiagramAsListOfChords}
  D
  \;=\;
  (i_1 j_1)
  (i_2 j_2)
  \cdots
  (i_d j_d)
  \,,
\end{equation}
possibly empty, and the product operation ``$\cdot$" on horizontal chord diagrams
is concatenation of these lists, the neutral element being given by the empty list.
For example:

\begin{equation}
  \label{MonoidOfHorizontalChordDiagrams}
 (ik) \cdot (i j)
 \;=\;
 (ik)(ij)
\phantom{AAAAA}
\begin{array}{c}
\left[
\!\!
\mbox{
\raisebox{-21pt}{
\begin{tikzpicture}

 \draw (-.23, .6) node {\tiny $i$};
 \draw (-.23+1.2, .6) node {\tiny $j$};
 \draw (-.23+2.4, .6) node {\tiny $k$};

 \clip (-1,+.5+1.5) rectangle (3.5,+.5);

 \draw[thick, orangeii]
   (0,+1)
   to
   node{\colorbox{white}{\hspace{-2pt}}}
   (2.4,+1);
 \begin{scope}[shift={(0,+1)}]
   \clip (0,-.1) rectangle (.1,+.1);
   \draw[draw=orangeii, fill=orangeii] (0,0) circle (.07);
 \end{scope}
 \begin{scope}[shift={(2.4,+1)}]
   \clip (-.1,-.1) rectangle (0,+.1);
   \draw[draw=orangeii, fill=orangeii] (0,0) circle (.07);
 \end{scope}

 \draw[thick, orangeii]
   (0,0)
   to
   (1.2,0);
 \begin{scope}[shift={(0,0)}]
   \clip (0,-.1) rectangle (.1,+.1);
   \draw[draw=orangeii, fill=orangeii] (0,0) circle (.07);
 \end{scope}
 \begin{scope}[shift={(1.2,0)}]
   \clip (-.1,-.1) rectangle (0,+.1);
   \draw[draw=orangeii, fill=orangeii] (0,0) circle (.07);
 \end{scope}

 \draw[ultra thick, darkblue] (0,2) to (0,-1);
 \draw[ultra thick, darkblue] (1.2,2) to (1.2,-1);
 \draw[ultra thick, darkblue] (2.4,2) to (2.4,-1);

 \draw (-.6,1.5) node {$\cdots$};
 \draw (0.6,1.5) node {$\cdots$};
 \draw (1.2+0.6,1.5) node {$\cdots$};
 \draw (2.4+0.6,1.5) node {$\cdots$};

 \draw (-.6,-.5) node {$\cdots$};
 \draw (0.6,-.5) node {$\cdots$};
 \draw (1.2+0.6,-.5) node {$\cdots$};
 \draw (2.4+0.6,-.5) node {$\cdots$};

\end{tikzpicture}
}
}
\!\!
\right]
\\
\cdot
\\
\left[
\!\!
\mbox{
\raisebox{-18pt}{
\begin{tikzpicture}

 \draw (-.23, 0.35) node {\tiny $i$};
 \draw (-.23+1.2, 0.35) node {\tiny $j$};
 \draw (-.23+2.4, 0.35) node {\tiny $k$};

 \clip (-1,+.5-1.5) rectangle (3.5,+.5);

 \draw[thick, orangeii]
   (0,+1)
   to
   node{\colorbox{white}{\hspace{-2pt}}}
   (2.4,+1);
 \begin{scope}[shift={(0,+1)}]
   \clip (0,-.1) rectangle (.1,+.1);
   \draw[draw=orangeii, fill=orangeii] (0,0) circle (.07);
 \end{scope}
 \begin{scope}[shift={(2.4,+1)}]
   \clip (-.1,-.1) rectangle (0,+.1);
   \draw[draw=orangeii, fill=orangeii] (0,0) circle (.07);
 \end{scope}

 \draw[thick, orangeii]
   (0,0)
   to
   (1.2,0);
 \begin{scope}[shift={(0,0)}]
   \clip (0,-.1) rectangle (.1,+.1);
   \draw[draw=orangeii, fill=orangeii] (0,0) circle (.07);
 \end{scope}
 \begin{scope}[shift={(1.2,0)}]
   \clip (-.1,-.1) rectangle (0,+.1);
   \draw[draw=orangeii, fill=orangeii] (0,0) circle (.07);
 \end{scope}

 \draw[ultra thick, darkblue] (0,2) to (0,-1);
 \draw[ultra thick, darkblue] (1.2,2) to (1.2,-1);
 \draw[ultra thick, darkblue] (2.4,2) to (2.4,-1);

 \draw (-.6,1.5) node {$\cdots$};
 \draw (0.6,1.5) node {$\cdots$};
 \draw (1.2+0.6,1.5) node {$\cdots$};
 \draw (2.4+0.6,1.5) node {$\cdots$};

 \draw (-.6,-.5) node {$\cdots$};
 \draw (0.6,-.5) node {$\cdots$};
 \draw (1.2+0.6,-.5) node {$\cdots$};
 \draw (2.4+0.6,-.5) node {$\cdots$};

\end{tikzpicture}
}
}
\!\!
\right]
\end{array}
\;\;\;=\;\;\;
\left[
\!\!
\mbox{
\raisebox{-40pt}{
\begin{tikzpicture}

 \draw (-.23, 1.9) node {\tiny $i$};
 \draw (-.23+1.2, 1.9) node {\tiny $j$};
 \draw (-.23+2.4, 1.9) node {\tiny $k$};

 \draw[thick, orangeii]
   (0,+1)
   to
   node{\colorbox{white}{\hspace{-2pt}}}
   (2.4,+1);
 \begin{scope}[shift={(0,+1)}]
   \clip (0,-.1) rectangle (.1,+.1);
   \draw[draw=orangeii, fill=orangeii] (0,0) circle (.07);
 \end{scope}
 \begin{scope}[shift={(2.4,+1)}]
   \clip (-.1,-.1) rectangle (0,+.1);
   \draw[draw=orangeii, fill=orangeii] (0,0) circle (.07);
 \end{scope}

 \draw[thick, orangeii]
   (0,0)
   to
   (1.2,0);
 \begin{scope}[shift={(0,0)}]
   \clip (0,-.1) rectangle (.1,+.1);
   \draw[draw=orangeii, fill=orangeii] (0,0) circle (.07);
 \end{scope}
 \begin{scope}[shift={(1.2,0)}]
   \clip (-.1,-.1) rectangle (0,+.1);
   \draw[draw=orangeii, fill=orangeii] (0,0) circle (.07);
 \end{scope}

 \draw[ultra thick, darkblue] (0,2) to (0,-1);
 \draw[ultra thick, darkblue] (1.2,2) to (1.2,-1);
 \draw[ultra thick, darkblue] (2.4,2) to (2.4,-1);

 \draw (-.6,1.5) node {$\cdots$};
 \draw (0.6,1.5) node {$\cdots$};
 \draw (1.2+0.6,1.5) node {$\cdots$};
 \draw (2.4+0.6,1.5) node {$\cdots$};

 \draw (-.6,-.5) node {$\cdots$};
 \draw (0.6,-.5) node {$\cdots$};
 \draw (1.2+0.6,-.5) node {$\cdots$};
 \draw (2.4+0.6,-.5) node {$\cdots$};

\end{tikzpicture}
}}
\!\!
\right]
\end{equation}

\noindent
{\bf (ii)}
The map that sends the chord
$(i j)$
to the permutation transposing the $i$th and $j$th of $N$
ordered elements
\begin{equation}
  \label{ATransposition}
  t_{i j}
  \;\coloneqq\;
  \left(
  \begin{array}{ccccccccc}
    1 & 2 & \cdots & i & \cdots & j & \cdots & N - 1 & N
    \\
    1 & 2 & \cdots & j & \cdots & i & \cdots & N -1  & N
  \end{array}
  \right)
\end{equation}
extends uniquely
to a monoid homomorphism from horizontal chord diagrams
\eqref{SetOfHorizontalChordDiagrams} to the symmetric group on $N$
elements:
\begin{equation}
  \label{MonoidHomomorphismFromHorizontalChordsToPermutations}
  \begin{tikzcd}[row sep=0pt]
    \mathcal{D}^{{}^{\mathrm{pb}}}_{\!N}
    \ar[
      rr,
      "\mathrm{perm}"
    ]
    &&
    \mathrm{Sym}(N)
    \\
    (i_1 j_1)
    \cdots
    (i_d j_d)
    &
      \longmapsto
    &
    t_{i_1 j_1}
    \circ
      \cdots
    \circ
    t_{i_d j_d}
  \end{tikzcd}
\end{equation}
A crucial role in the following discussion is played by the number of cycles
in the permutation underlying a chord diagram:
\vspace{-2mm}
\begin{equation}
  \label{CyclesInPermutationCorrespondingToChordDiagram}
    \xymatrix@C=50pt{
      \overset{
        \mathclap{
        \mbox{
          \tiny
          \color{darkblue}
          \bf
          \begin{tabular}{c}
            set of
            \\
            horizontal
            chord diagrams
            \\
            with $N$ strands
            \\
            $\phantom{a}$
          \end{tabular}
        }
        }
      }{
        \mathcal{D}^{{}^{\mathrm{pb}}}_{N}
      }
      \ar[rr]^-{ \mathrm{perm} }_-{
        \mathclap{
        \mbox{
          \tiny
          \color{greenii}
          \bf
          \begin{tabular}{c}
            take chord $t_{i j}$
            \\
            to transposition $i \leftrightarrow j$
            \\
            and consecutive chords to
            \\
            composition of transpositions
          \end{tabular}
        }
        }
      }
      &&
      \overset{
        \mathclap{
        \mbox{
          \tiny
          \color{darkblue}
          \bf
          \begin{tabular}{c}
            set of
            \\
            permutations of
            \\
            $N$ elements
            \\
            $\phantom{a}$
          \end{tabular}
        }
        }
      }{
        \mathrm{Sym}(N)
      }
      \ar[rr]^-{ \#\mathrm{cycles} }_-{
        \mbox{
          \tiny
          \color{greenii}
          \bf
          \begin{tabular}{c}
            number of
            \\
            cycles (orbits)
            \\
            of a permutation
          \end{tabular}
        }
      }
      &&
      \{1,\cdots, N\}
      \subset
      \mathbb{N}
    }
\end{equation}

\noindent
{\bf (iii)}
The {\it algebra of horizontal chord diagrams}
\begin{equation}
  \label{AlgebraOfHorizontalChordDiagrams}
  \mathcal{A}^{{}^{\mathrm{pb}}}_{N}
  \;:=\;
  \mathbb{C}[\mathcal{D}^{\mathrm{pb}}_{\!N}]/(\mathrm{2T}, \mathrm{4T})
\end{equation}
is the associative unital algebra, graded by number of chords,
which is spanned by the monoid of horizontal chord diagrams \eqref{SetOfHorizontalChordDiagrams}
and then quotiented by the ideal generated by:

\noindent {\bf (a)} the {\it 2T relations}:

\vspace{-2mm}
\begin{equation}
\label{2TRelationsOnHorizontalChordDiagrams}
\scalebox{.9}{
\begin{tabular}{ccc}
$
\left[
\raisebox{-54pt}{
\begin{tikzpicture}

 \draw[thick, orangeii] (0,1) to (1.2,1);
 \begin{scope}[shift={(0,1)}]
   \clip (0,-.1) rectangle (.1,+.1);
   \draw[draw=orangeii, fill=orangeii] (0,0) circle (.07);
 \end{scope}
 \begin{scope}[shift={(1.2,1)}]
   \clip (-.1,-.1) rectangle (0,+.1);
   \draw[draw=orangeii, fill=orangeii] (0,0) circle (.07);
 \end{scope}

 \draw[thick, orangeii] (2.4,0) to (3.6,0);
 \begin{scope}[shift={(2.4,0)}]
   \clip (0,-.1) rectangle (.1,+.1);
   \draw[draw=orangeii, fill=orangeii] (0,0) circle (.07);
 \end{scope}
 \begin{scope}[shift={(3.6,0)}]
   \clip (-.1,-.1) rectangle (0,+.1);
   \draw[draw=orangeii, fill=orangeii] (0,0) circle (.07);
 \end{scope}

 \draw[ultra thick, darkblue] (0,2) to (0,-1);
 \draw[ultra thick, darkblue] (1.2,2) to (1.2,-1);
 \draw[ultra thick, darkblue] (2.4,2) to (2.4,-1);
 \draw[ultra thick, darkblue] (3.6,2) to (3.6,-1);

 \draw (-.6,1.5) node {$\cdots$};
 \draw (.6,1.5) node {$\cdots$};
 \draw (1.2+.6,1.5) node {$\cdots$};
 \draw (2.4+.6,1.5) node {$\cdots$};
 \draw (3.6+.6,1.5) node {$\cdots$};

 \draw (-.6,-.5) node {$\cdots$};
 \draw (.6,-.5) node {$\cdots$};
 \draw (1.2+.6,-.5) node {$\cdots$};
 \draw (2.4+.6,-.5) node {$\cdots$};
 \draw (3.6+.6,-.5) node {$\cdots$};

 \draw (0,-1.4) node {\small $i$};
 \draw (1.2,-1.4) node {\small $j$};
 \draw (2.4,-1.4) node {\small $k$};
 \draw (3.6,-1.4) node {\small $l$};

\end{tikzpicture}
}
\right]
$
&
 $\sim$
&
$
\left[
\raisebox{-54pt}{
\begin{tikzpicture}

 \draw[thick, orangeii] (0,0) to (1.2,0);
 \begin{scope}[shift={(0,0)}]
   \clip (0,-.1) rectangle (.1,+.1);
   \draw[draw=orangeii, fill=orangeii] (0,0) circle (.07);
 \end{scope}
 \begin{scope}[shift={(1.2,0)}]
   \clip (-.1,-.1) rectangle (0,+.1);
   \draw[draw=orangeii, fill=orangeii] (0,0) circle (.07);
 \end{scope}

 \draw[thick, orangeii] (2.4,1) to (3.6,1);
 \begin{scope}[shift={(2.4,1)}]
   \clip (0,-.1) rectangle (.1,+.1);
   \draw[draw=orangeii, fill=orangeii] (0,0) circle (.07);
 \end{scope}
 \begin{scope}[shift={(3.6,1)}]
   \clip (-.1,-.1) rectangle (0,+.1);
   \draw[draw=orangeii, fill=orangeii] (0,0) circle (.07);
 \end{scope}

 \draw[ultra thick, darkblue] (0,2) to (0,-1);
 \draw[ultra thick, darkblue] (1.2,2) to (1.2,-1);
 \draw[ultra thick, darkblue] (2.4,2) to (2.4,-1);
 \draw[ultra thick, darkblue] (3.6,2) to (3.6,-1);

 \draw (-.6,1.5) node {$\cdots$};
 \draw (.6,1.5) node {$\cdots$};
 \draw (1.2+.6,1.5) node {$\cdots$};
 \draw (2.4+.6,1.5) node {$\cdots$};
 \draw (3.6+.6,1.5) node {$\cdots$};

 \draw (-.6,-.5) node {$\cdots$};
 \draw (.6,-.5) node {$\cdots$};
 \draw (1.2+.6,-.5) node {$\cdots$};
 \draw (2.4+.6,-.5) node {$\cdots$};
 \draw (3.6+.6,-.5) node {$\cdots$};

 \draw (0,-1.4) node {\small $i$};
 \draw (1.2,-1.4) node {\small $j$};
 \draw (2.4,-1.4) node {\small $k$};
 \draw (3.6,-1.4) node {\small $l$};

\end{tikzpicture}
}
\right]
$
\end{tabular}
}
\end{equation}

\noindent {\bf (b)} the {\it 4T relations}:

\vspace{-.3cm}

\begin{equation}
\label{4TRelationOnHorizontalChordDiagrams}
\scalebox{.85}{
\hspace{-1.3cm}
\begin{tabular}{cccccccc}
&
$
\left[\!\!\!\!\!
\raisebox{-54pt}{
\begin{tikzpicture}

 \draw[thick, orangeii]
   (0,+1)
   to
   node{\colorbox{white}{\hspace{-2pt}}}
   (2.4,+1);
 \begin{scope}[shift={(0,+1)}]
   \clip (0,-.1) rectangle (.1,+.1);
   \draw[draw=orangeii, fill=orangeii] (0,0) circle (.07);
 \end{scope}
 \begin{scope}[shift={(2.4,+1)}]
   \clip (-.1,-.1) rectangle (0,+.1);
   \draw[draw=orangeii, fill=orangeii] (0,0) circle (.07);
 \end{scope}

 \draw[thick, orangeii]
   (0,0)
   to
   (1.2,0);
 \begin{scope}[shift={(0,0)}]
   \clip (0,-.1) rectangle (.1,+.1);
   \draw[draw=orangeii, fill=orangeii] (0,0) circle (.07);
 \end{scope}
 \begin{scope}[shift={(1.2,0)}]
   \clip (-.1,-.1) rectangle (0,+.1);
   \draw[draw=orangeii, fill=orangeii] (0,0) circle (.07);
 \end{scope}

 \draw[ultra thick, darkblue] (0,2) to (0,-1);
 \draw[ultra thick, darkblue] (1.2,2) to (1.2,-1);
 \draw[ultra thick, darkblue] (2.4,2) to (2.4,-1);

 \draw (-.6,1.5) node {$\cdots$};
 \draw (0.6,1.5) node {$\cdots$};
 \draw (1.2+0.6,1.5) node {$\cdots$};
 \draw (2.4+0.6,1.5) node {$\cdots$};

 \draw (-.6,-.5) node {$\cdots$};
 \draw (0.6,-.5) node {$\cdots$};
 \draw (1.2+0.6,-.5) node {$\cdots$};
 \draw (2.4+0.6,-.5) node {$\cdots$};

 \draw (0,-1.4) node {$i$};
 \draw (1.2,-1.4) node {$j$};
 \draw (2.4,-1.4) node {$k$};

\end{tikzpicture}
}
\!\!\!\!\! \right]
$
&
\!\!\!\!\!\!\!$+$\!\!\!\!\!\!\!
&
$
\left[ \!\!\!\!\!
\raisebox{-54pt}{
\begin{tikzpicture}

 \draw[thick, orangeii]
   (1.2,+1)
   to
   (2.4,+1);
 \begin{scope}[shift={(1.2,+1)}]
   \clip (0,-.1) rectangle (.1,+.1);
   \draw[draw=orangeii, fill=orangeii] (0,0) circle (.07);
 \end{scope}
 \begin{scope}[shift={(2.4,+1)}]
   \clip (-.1,-.1) rectangle (0,+.1);
   \draw[draw=orangeii, fill=orangeii] (0,0) circle (.07);
 \end{scope}

 \draw[thick, orangeii]
   (0,0)
   to
   (1.2,0);
 \begin{scope}[shift={(0,0)}]
   \clip (0,-.1) rectangle (.1,+.1);
   \draw[draw=orangeii, fill=orangeii] (0,0) circle (.07);
 \end{scope}
 \begin{scope}[shift={(1.2,0)}]
   \clip (-.1,-.1) rectangle (0,+.1);
   \draw[draw=orangeii, fill=orangeii] (0,0) circle (.07);
 \end{scope}

 \draw[ultra thick, darkblue] (0,2) to (0,-1);
 \draw[ultra thick, darkblue] (1.2,2) to (1.2,-1);
 \draw[ultra thick, darkblue] (2.4,2) to (2.4,-1);

 \draw (-.6,1.5) node {$\cdots$};
 \draw (0.6,1.5) node {$\cdots$};
 \draw (1.2+0.6,1.5) node {$\cdots$};
 \draw (2.4+0.6,1.5) node {$\cdots$};

 \draw (-.6,-.5) node {$\cdots$};
 \draw (0.6,-.5) node {$\cdots$};
 \draw (1.2+0.6,-.5) node {$\cdots$};
 \draw (2.4+0.6,-.5) node {$\cdots$};

 \draw (0,-1.4) node {$i$};
 \draw (1.2,-1.4) node {$j$};
 \draw (2.4,-1.4) node {$k$};

\end{tikzpicture}
}
\!\!\!\!\! \right]
$
&
\!\!\!\!\!\!\! $\sim$ \!\!\!\!\!\!\!
&
$
\left[ \!\!\!\!\!
\raisebox{-54pt}{
\begin{tikzpicture}

 \draw[thick, orangeii]
   (0,0)
   to
   node{\colorbox{white}{\hspace{-2pt}}}
   (2.4,0);
 \begin{scope}[shift={(0,0)}]
   \clip (0,-.1) rectangle (.1,+.1);
   \draw[draw=orangeii, fill=orangeii] (0,0) circle (.07);
 \end{scope}
 \begin{scope}[shift={(2.4,0)}]
   \clip (-.1,-.1) rectangle (0,+.1);
   \draw[draw=orangeii, fill=orangeii] (0,0) circle (.07);
 \end{scope}

 \draw[thick, orangeii]
   (0,1)
   to
   (1.2,1);
 \begin{scope}[shift={(0,1)}]
   \clip (0,-.1) rectangle (.1,+.1);
   \draw[draw=orangeii, fill=orangeii] (0,0) circle (.07);
 \end{scope}
 \begin{scope}[shift={(1.2,1)}]
   \clip (-.1,-.1) rectangle (0,+.1);
   \draw[draw=orangeii, fill=orangeii] (0,0) circle (.07);
 \end{scope}

 \draw[ultra thick, darkblue] (0,2) to (0,-1);
 \draw[ultra thick, darkblue] (1.2,2) to (1.2,-1);
 \draw[ultra thick, darkblue] (2.4,2) to (2.4,-1);

 \draw (-.6,1.5) node {$\cdots$};
 \draw (0.6,1.5) node {$\cdots$};
 \draw (1.2+0.6,1.5) node {$\cdots$};
 \draw (2.4+0.6,1.5) node {$\cdots$};

 \draw (-.6,-.5) node {$\cdots$};
 \draw (0.6,-.5) node {$\cdots$};
 \draw (1.2+0.6,-.5) node {$\cdots$};
 \draw (2.4+0.6,-.5) node {$\cdots$};

 \draw (0,-1.4) node {$i$};
 \draw (1.2,-1.4) node {$j$};
 \draw (2.4,-1.4) node {$k$};

\end{tikzpicture}
}
\!\!\!\!\! \right]
$
&
\!\!\!\!\!\!\!$+$\!\!\!\!\!\!\!
&
$
\left[\!\!\!\!\!
\raisebox{-55pt}{
\begin{tikzpicture}

 \draw[thick, orangeii]
   (1.2,0)
   to
   (2.4,0);
 \begin{scope}[shift={(1.2,0)}]
   \clip (0,-.1) rectangle (.1,+.1);
   \draw[draw=orangeii, fill=orangeii] (0,0) circle (.07);
 \end{scope}
 \begin{scope}[shift={(2.4,0)}]
   \clip (-.1,-.1) rectangle (0,+.1);
   \draw[draw=orangeii, fill=orangeii] (0,0) circle (.07);
 \end{scope}

 \draw[thick, orangeii]
   (0,1)
   to
   (1.2,1);
 \begin{scope}[shift={(0,1)}]
   \clip (0,-.1) rectangle (.1,+.1);
   \draw[draw=orangeii, fill=orangeii] (0,0) circle (.07);
 \end{scope}
 \begin{scope}[shift={(1.2,1)}]
   \clip (-.1,-.1) rectangle (0,+.1);
   \draw[draw=orangeii, fill=orangeii] (0,0) circle (.07);
 \end{scope}

 \draw[ultra thick, darkblue] (0,2) to (0,-1);
 \draw[ultra thick, darkblue] (1.2,2) to (1.2,-1);
 \draw[ultra thick, darkblue] (2.4,2) to (2.4,-1);

 \draw (-.6,1.5) node {$\cdots$};
 \draw (0.6,1.5) node {$\cdots$};
 \draw (1.2+0.6,1.5) node {$\cdots$};
 \draw (2.4+0.6,1.5) node {$\cdots$};

 \draw (-.6,-.5) node {$\cdots$};
 \draw (0.6,-.5) node {$\cdots$};
 \draw (1.2+0.6,-.5) node {$\cdots$};
 \draw (2.4+0.6,-.5) node {$\cdots$};

 \draw (0,-1.4) node {$i$};
 \draw (1.2,-1.4) node {$j$};
 \draw (2.4,-1.4) node {$k$};

\end{tikzpicture}
}
\!\!\!\!\!\right]
$
\end{tabular}
}
\end{equation}

\noindent
{\bf (iv)} The complex vector
space of {\it weight systems on horizontal chord diagrams}
is the graded linear dual space
$$
  \mathcal{W}^{{}^{\mathrm{pb}}}_{N}
  \;:=\;
  \big(
    \mathcal{A}^{{}^{\mathrm{pb}}}_{N}
  \big)^\ast
$$
to \eqref{AlgebraOfHorizontalChordDiagrams};
hence a weight system is a complex-linear map
  \begin{equation}
    \label{AWeightSystem}
    w \;:\;
    \mathcal{A}^{{}^{\mathrm{pb}}}_{N}
    \longrightarrow
    \mathbb{C}
  \end{equation}
and is of degree $- d$ if it is supported on chord diagrams of degree $d$.
\end{defn}

\begin{remark}[Dependence on algebra structure]
  The definition of weight systems \eqref{AWeightSystem}
  according to Def. \ref{AlgebraOfAndWeightSystemsOnHorizontalChordDiagrams}
  does not depend on the algebra structure on the
  space of chord diagrams \eqref{AlgebraOfHorizontalChordDiagrams},
  but the specialization of weight systems to
  quantum states (Def. \ref{States}) does.
\end{remark}

\medskip

\noindent
{\bf Fundamental Lie algebra weight systems.}
Recall from \cite{BarNatan96} (reviewed in \cite[\S 3.4]{SS19b}) that
the main source of weight systems on horizontal chord diagrams
(Def. \ref{AlgebraOfAndWeightSystemsOnHorizontalChordDiagrams})
are metric Lie representations
$\rho : \mathfrak{g}\otimes \mathbf{V} \xrightarrow{\;} \mathbf{V}$
of metric Lie algebras $\mathfrak{g}$,
\begin{equation}
  \label{LieAlgebraWeightSystems}
  \xymatrix{
    \mathrm{MetricLieModules}
    \ar[rr]^-{w_{(-)}}
    &&
    \mathcal{W}^{{}^{\mathrm{pb}}}_{N}
    \,,
  }
\end{equation}
where the {\it Lie algebra weight system} $w_{(\mathfrak{g}, \mathbf{V})}$ sends a chord diagram
$D \in \mathcal{D}^{{}^{\mathrm{pb}}}_{N}$ to
(see \eqref{EvaluationOfAFundamentalWeightSystem} for illustration):
the number
obtained by labelling all strands by $V$, all vertices by $\rho$,
all chords by $\mathfrak{g}$, then closing all strands to circles using the metric,
regarding the result as Penrose notation 
(review in \cite{Selinger09})
for a rank-0 tensor in the category
of finite-dimensional complex vector spaces,
and evaluating it as such to a complex number.

%(More generally, one may apply a permutation to the set of strands
%before closing them to circles, but we do not consider this here.)

\begin{defn}[Fundamental $\mathfrak{gl}(n)$-weight system]
  \label{FundamentalWeightSystemOfgl2}
  For $n \in \mathbb{N}_+$, we write
  \begin{equation}
    \label{gl2FundamentalWeightSystem}
    \mathrm{w}_{(\mathfrak{gl}(n),\mathbf{n})}
    \;:\;
    \mathcal{A}^{{}^{\mathrm{pb}}}_{N}
    \longrightarrow
    \mathbb{C}
  \end{equation}
  for the normalized
  Lie algebra weight system \eqref{LieAlgebraWeightSystems} induced by
  the defining complex $n$-dimensional Lie representation $\mathbf{n}$
  of the general complex-linear Lie algebra $\mathfrak{gl}(n)$
  equipped with the metric given by the trace in $\mathbf{n} \simeq_{\mathbb{C}} \mathbb{C}^n$.
\end{defn}

\begin{example}[Fundamental metric on $\mathfrak{gl}(2)$]
  \label{FundamentalMetricOngl2}
  For $n = 2$
  and with respect to the complex linear basis
  of $\mathfrak{gl}(2)$ given by
  \begin{equation}
    \label{LinearBasisForgl2}
    \left\{
    x_0
    :=
    \left[\!\!
      \begin{array}{cc}
        1 & 0
        \\
        0 & 1
      \end{array}
    \!\!\right]
    ,\;
    x_1
    :=
    \left[\!\!
      \begin{array}{cc}
        1 & 0
        \\
        0 & -1
      \end{array}
   \!\! \right]
    ,\;
    x_+
    :=
    \left[\!\!
      \begin{array}{cc}
        0 & 0
        \\
        1 & 0
      \end{array}
    \!\!\right]
    ,\;
    x_-
    :=
    \left[\!\!
      \begin{array}{cc}
        0 & 1
        \\
        0 & 0
      \end{array}
   \!\! \right]
    \right\}
  \end{equation}
  the metric on $\mathfrak{gl}(2)$ according to Def. \ref{FundamentalWeightSystemOfgl2}
  has components
  \vspace{-2mm}
  \begin{equation}
    \label{gl2MetricComponents}
    \big(
      g_{i j}
      :=
      g(x_i, x_j)
    \big)_{i,j \in \{0,1,+,-\}}
    \;=\;
    \left[
    \begin{array}{cccc}
      2 & 0 & 0 & 0
      \\
      0 & 2 & 0 & 0
      \\
      0 & 0 & 0 & 1
      \\
      0 & 0 & 1 & 0
    \end{array}
    \right]
    \,.
  \end{equation}
\end{example}

\begin{lemma}[Fundamental weights and braiding {\cite[Fact 6]{BarNatan96}}]
  \label{gl2ValueOfFundamentalWeightSystem}
  The fundamental metric on $\mathfrak{gl}(2)$ (Ex. \ref{FundamentalMetricOngl2})
  has the special property that
  it makes the value of the fundamental $\mathfrak{gl}(2)$-weight system \eqref{gl2FundamentalWeightSystem}
  on a single chord
  be the braiding operation:
\begin{equation}
  \label{gl2ChordIsBraiding}
\mbox{
\begin{tabular}{|c|l|}
  \hline
  \raisebox{19pt}{
  \begin{tabular}{c}
    \bf Metric
    \\
    \bf Lie algebra
    \\
    \raisebox{-6pt}{
      $(\mathfrak{g}, g)$
    }
  \end{tabular}
  }
  &
  $
  \overset{
    \scalebox{1}{
      \begin{tabular}{c}
        \bf Metric contraction of fundamental action tensors
      \end{tabular}
    }
  }{
  \mbox{
  \hspace{-5.2cm}
  \begin{rotate}{90}
    $\mathclap{
      \scalebox{.5}{
        \color{darkblue}
        \begin{tabular}{c}
          fundamental
          \\
          representation
        \end{tabular}
      }
    }$
  \end{rotate}
  \hspace{2pt}
  \raisebox{-26pt}{
 \begin{tikzpicture}
 \draw (-1,.75) node {\scalebox{.8}{$V$}};
 \draw (+1,.75) node {\scalebox{.8}{$V$}};
 \draw (-1,-.75) node {\scalebox{.8}{$V$}};
 \draw (+1,-.75) node {\scalebox{.8}{$V$}};
 \begin{scope}
   \clip (-1,.5) rectangle (1,-.5);
   \draw[draw=orangeii, fill=orangeii] (-1,0) circle (.07);
   \draw[draw=orangeii, fill=orangeii] (+1,0) circle (.07);
 \end{scope}
 \draw[thick, orangeii]
   (-1,0)
     to
     node[above]
     {
       \scalebox{.5}{
         \color{darkblue}
         virtual gluon line
       }
     }
   (1,0);
 \draw[ultra thick, darkblue] (-1,.5) to (-1,-.5);
 \draw[ultra thick, darkblue] (+1,.5)
   to
     node[very near start]
       { $
         \mathrlap{
         \;
         \scalebox{.5}{
           \color{darkblue} quark line
         }
         }
         $
       }
   (+1,-.5);
 \begin{scope}[shift={(-1,0)}]
 \draw[draw=darkblue, thick, fill=white]
   (-.15,-.13)
   rectangle node{\tiny $\rho$}
   (.15,.17);
 \end{scope}
 \begin{scope}[shift={(+1,0)}]
 \draw[draw=darkblue, thick, fill=white]
   (-.15,-.13)
   rectangle node{\tiny $\rho$}
   (.15,.17);
 \end{scope}
\end{tikzpicture}
  }
  }
  \mathrlap{
  \;\;\;\;\;\;\;
  =
  g^{i j} x_i \otimes x_j
  \;\;\in\; \mathrm{End}(V \otimes V)
  }
  }
  $
  \\
  \hline
  \hline
  $\big( \mathfrak{gl}(2), \mathrm{tr}_{\mathbf{2}}(-\cdot -)\big)$
  &
  \hspace{1.2cm}
\raisebox{-25pt}{
\begin{tikzpicture}

  \draw[ultra thick, darkblue]
    (-.9,-.4) .. controls (-.4,-.4) and (+.4,+.4) .. (+.9,+.4);

  \begin{scope}[shift={(.03,.03)}]
  \draw[ultra thick, white]
    (-.9,.4) .. controls (-.4,.4) and (+.4,-.4) .. (+.9,-.4);
  \end{scope}

  \begin{scope}[shift={(-.03,-.03)}]
  \draw[ultra thick, white]
    (-.9,.4) .. controls (-.4,.4) and (+.4,-.4) .. (+.9,-.4);
  \end{scope}
  \draw[ultra thick, darkblue]
    (-.9,.4) .. controls (-.4,.4) and (+.4,-.4) .. (+.9,-.4);

  \draw[ultra thick, darkblue]
    (-1,.5) ..  controls (-1,.4) and (-1,.4) .. (-.9,.4);
  \draw[ultra thick, darkblue]
    (+1,.5) ..  controls (+1,.4) and (+1,.4) .. (+.9,.4);
  \draw[ultra thick, darkblue]
    (-1,-.5) ..  controls (-1,-.4) and (-1,-.4) .. (-.9,-.4);
  \draw[ultra thick, darkblue]
    (+1,-.5) ..  controls (+1,-.4) and (+1,-.4) .. (+.9,-.4);

 \draw (-1,.75) node {\scalebox{.8}{$\mathclap{\phantom{\vert^{\vert^\vert}}}V$}};
 \draw (+1,.75) node {\scalebox{.8}{$\mathclap{\phantom{\vert_{\vert_\vert}}}V$}};
 \draw (-1,-.75) node {\scalebox{.8}{$V$}};
 \draw (+1,-.75) node {\scalebox{.8}{$V$}};

\end{tikzpicture}
}
$
  =\;
  \begin{aligned}
    & \phantom{=} \phantom{+}
    \tfrac{1}{2} x_0 \otimes x_0
    +
    \tfrac{1}{2} x_1 \otimes x_1
    \\
    &
    \phantom{=}
    +
    x_+ \otimes x_-
    +
    x_- \otimes x_+
  \end{aligned}
$
\\
\hline
\end{tabular}
}
\end{equation}
\end{lemma}
\begin{proof}
  By explicit computation
  in the canonical linear basis \eqref{LinearBasisForgl2}
  with its metric components \eqref{gl2MetricComponents}.
  For example:
  $$
    \begin{aligned}
    g^{i j} x_i \otimes x_j
    \left(
    \left[\!\!
    \begin{array}{c}
      1
      \\
      0
    \end{array}
   \!\! \right]
    \otimes
    \left[\!\!
    \begin{array}{c}
      0
      \\
      1
    \end{array}
   \!\! \right]
    \right)
    & =
    \underset{= 0}{
    \underbrace{
    \tfrac{1}{2}
    \left[\!\!
    \begin{array}{c}
      1
      \\
      0
    \end{array}
    \!\! \right]
    \otimes
    \left[\!\!
    \begin{array}{c}
      0
      \\
      1
    \end{array}
   \!\! \right]
    \;+\;
    \tfrac{1}{2}
    \left[ \!\!
    \begin{array}{c}
      1
      \\
      0
    \end{array}
    \!\! \right]
    \otimes
    \left[\!\!
    \begin{array}{c}
      0
      \\
      -1
    \end{array}
    \!\! \right]
    }}
    \;+\;
    \left[\!\!
    \begin{array}{c}
      0
      \\
      1
    \end{array}
    \!\! \right]
    \otimes
    \left[\!\!
    \begin{array}{c}
      1
      \\
      0
    \end{array}
    \!\! \right]
    +
    \left[\!\!
    \begin{array}{c}
      0
      \\
      0
    \end{array}
    \!\!\right]
    \otimes
    \left[\!\!
    \begin{array}{c}
      0
      \\
      0
    \end{array}
    \!\! \right]
    \mathrlap{\,.}
%    \\
%    & = \phantom{+ \tfrac{1}{2}}
%    \left[
%    \begin{array}{c}
%      0
%      \\
%      1
%    \end{array}
%    \right]
%    \otimes
%    \left[
%    \begin{array}{c}
%      1
%      \\
%      0
%    \end{array}
%    \right]
    \end{aligned}
  $$
%  and
%  $$
%    \begin{aligned}
%    g^{i j} x_i \otimes x_j
%    \left(
%    \left[
%    \begin{array}{c}
%      0
%      \\
%      1
%    \end{array}
%    \right]
%    \otimes
%    \left[
%    \begin{array}{c}
%      1
%      \\
%      0
%    \end{array}
%    \right]
%    \right)
%    & = \phantom{+}
%    \tfrac{1}{2}
%    \left[
%    \begin{array}{c}
%      0
%      \\
%      1
%    \end{array}
%    \right]
%    \otimes
%    \left[
%    \begin{array}{c}
%      1
%      \\
%      0
%    \end{array}
%    \right]
%    +
%    \tfrac{1}{2}
%    \left[
%    \begin{array}{c}
%      0
%      \\
%      -1
%    \end{array}
%    \right]
%    \otimes
%    \left[
%    \begin{array}{c}
%      1
%      \\
%      0
%    \end{array}
%    \right]
%    \\
%    & \phantom{=} \hspace{1.4pt} +
%    \phantom{\tfrac{1}{2}}
%    \left[
%    \begin{array}{c}
%      0
%      \\
%      0
%    \end{array}
%    \right]
%    \otimes
%    \left[
%    \begin{array}{c}
%      0
%      \\
%      0
%    \end{array}
%    \right]
%    +
%    \phantom{\tfrac{1}{2}}
%    \left[
%    \begin{array}{c}
%      1
%      \\
%      0
%    \end{array}
%    \right]
%    \otimes
%    \left[
%    \begin{array}{c}
%      0
%      \\
%      1
%    \end{array}
%    \right]
%    \\
%    & = \phantom{+ \tfrac{1}{2}}
%    \left[
%    \begin{array}{c}
%      1
%      \\
%      0
%    \end{array}
%    \right]
%    \otimes
%    \left[
%    \begin{array}{c}
%      0
%      \\
%      1
%    \end{array}
%    \right]
%    \end{aligned}
%  $$
\vspace{-.9cm}

\end{proof}

\begin{cor}[Fundamental weight systems in terms of permutation cycles
({\cite[Prop. 2.1]{BarNatan96}})]
  \label{FundamentalWeightSystemInTermsOfCycles}
  The value of the fundamental $\mathfrak{gl}(2)$-weight system
  $w_{(\mathfrak{gl}(2),\mathbf{2})}$ \eqref{gl2FundamentalWeightSystem} on
  a horizontal chord diagram $D \in \mathcal{A}^{{}^{\mathrm{pb}}}_{N}$
  equals, up to normalization,
  2 taken to the power of the number
  of cycles \eqref{CyclesInPermutationCorrespondingToChordDiagram}
  in the permutation $\mathrm{perm}(D)$
  \eqref{MonoidHomomorphismFromHorizontalChordsToPermutations}
  corresponding to the chord diagram:
  \begin{equation}
    \label{ValueOfFundamentalgl2WeightSystemOnChordDiagram}
    \begin{aligned}
      w_{(\mathfrak{gl}(2), \mathbf{2})}([D])
      &
      \;=\;
      \overset{
       \mathclap{
        \mbox{
          \tiny
          \color{darkblue}
          \bf
          normalization
        }
        }
      }{
      \overbrace{
        2^{-N}
      }
      }
      \cdot 2^{\#\mathrm{cycles}(\mathrm{perm}(D))}
      \\
      & =
      e^{
        \mathrm{ln}(2)
          \cdot
        \big(
          \#\mathrm{cycles}(\mathrm{perm}(D))
          -
          N
        \big)
      }.
    \end{aligned}
  \end{equation}

  \noindent
  Generally, the analogous statement is true for the fundamental
  $\mathfrak{gl}(n)$-weight system (Def. \ref{FundamentalWeightSystemOfgl2})
  for all $n \in \mathbb{N}$, $n \geq 2$:
  \begin{equation}
    \label{ValueOfFundamentalgl2WeightSystemOnChordDiagram}
    w_{(\mathfrak{gl}(n),\mathbf{n})}([D])
    \;=\;
    e^{
      \mathrm{ln}(n)
        \cdot
      \big(
        \#\mathrm{cycles}(\mathrm{perm}(D))
        -
        N
      \big)
    }
    \,.
  \end{equation}

\end{cor}

For example:
\vspace{-3mm}
\begin{equation}
  \label{EvaluationOfAFundamentalWeightSystem}
  \mbox{
   \hspace{-1cm}
  \begin{tabular}{ccc}
  \raisebox{-45pt}{
    \begin{tikzpicture}

 \draw[draw=white, fill=white] (-4.5,3) rectangle (.15+.5,-1.8);

 \draw[thick, orangeii] (-4,1) to (-2,1);
 \draw[thick, orangeii] (-2,0) to (0,0);
 \draw[thick, orangeii] (-4,-1) to (0,-1);

 \draw[thick, white] (-2-.055,-1) to (-2+.055,-1);

 \begin{scope}[shift={(-4,1)}]
   \clip (0,-.1) rectangle (0+.1,+.1);
   \draw[draw=orangeii, fill=orangeii] (0,0) circle (.07);
 \end{scope}

 \begin{scope}[shift={(-2,1)}]
   \clip (-.1,-.1) rectangle (0,+.1);
   \draw[draw=orangeii, fill=orangeii] (0,0) circle (.07);
 \end{scope}

 \begin{scope}[shift={(-4,-1)}]
   \clip (0,-.1) rectangle (0+.1,+.1);
   \draw[draw=orangeii, fill=orangeii] (0,0) circle (.07);
 \end{scope}

 \begin{scope}[shift={(0,-1)}]
   \clip (-.1,-.1) rectangle (0,+.1);
   \draw[draw=orangeii, fill=orangeii] (0,0) circle (.07);
 \end{scope}

 \begin{scope}[shift={(-2,0)}]
   \clip (0,-.1) rectangle (0+.1,+.1);
   \draw[draw=orangeii, fill=orangeii] (0,0) circle (.07);
 \end{scope}

 \begin{scope}[shift={(0,0)}]
   \clip (-.1,-.1) rectangle (0,+.1);
   \draw[draw=orangeii, fill=orangeii] (0,0) circle (.07);
 \end{scope}

 \draw[darkblue, ultra thick] (-4,1.8) to (-4,-1.8);
 \draw[darkblue, ultra thick] (-2,1.8) to (-2,-1.8);
 \draw[darkblue, ultra thick] (0,1.8) to (0,-1.8);

 \draw (-4+.2,-1.5) node {\tiny $1$};
 \draw (-2+.2,-1.5) node {\tiny $2$};
 \draw (-0+.2,-1.5) node {\tiny $3$};

 \draw[->] (-4.3, -.8+.4) to (-4.3, -1.8+.4);

\end{tikzpicture}
}
  &
  \hspace{-1.4cm}
  $
    \underset{
      \mathclap{
      \mbox{
        \tiny
        \color{darkblue}
        \begin{tabular}{c}
          fundamental \\
          $\mathfrak{gl}(n)$-weight
          \\
          system
        \end{tabular}
      }
      }
    }{
      \overset{
        n^3
        \cdot
        w_{(\mathfrak{gl}(n),\mathbf{n})}(-)
      }{
        \xmapsto{\;\;\;\;\;\;\;\;\;\;\;\;\;\;\;\;\;\;\;\;\;\;\;\;\;\;}
      }
    }
  $
  &
  \raisebox{-90pt}{
\begin{tikzpicture}

 \draw[black, thick, dashed] (.5+.67,.65) ellipse (.8 and 3.2);

 \draw[draw=white, fill=white] (-4.5,3) rectangle (.15+.5,-1.8);

 \draw[thick, white] (-2-.055,-1) to (-2+.055,-1);

 \draw (-4-.2, {3-(3+1.8)/3}) node {\scalebox{.6}{$\mathbf{n}$}};
 \draw (-2-.2, {3-(3+1.8)/3}) node {\scalebox{.6}{$\mathbf{n}$}};
 \draw (-0-.2, {3-(3+1.8)/3}) node {\scalebox{.6}{$\mathbf{n}$}};

 \draw[orangeii, thick] (-4,{3-.5*(3+1.8)/3}) to (-2,{3-.5*(3+1.8)/3});

 \draw[orangeii, thick] (-2,{3-1.5*(3+1.8)/3}) to (-0,{3-1.5*(3+1.8)/3});

 \draw[orangeii, thick] (-4,{3-2.5*(3+1.8)/3}) to (-0,{3-2.5*(3+1.8)/3});

 \draw[white, line width=3pt]
   (-2,3) to (-2,-1.8);

 \draw[darkblue, ultra thick]
   (-4,3) to (-4,-1.8);
 \draw[darkblue, ultra thick]
   (-2,3) to (-2,-1.8);
 \draw[darkblue, ultra thick]
   (-0,3) to (-0,-1.8);

  \begin{scope}[shift={(-4,{3-0.5*(3+1.8)/3})}]
    \draw[draw=orangeii, fill=orangeii] (0,0) circle (.07);
   \draw[draw=darkblue, thick, fill=white]
     (-.15,-.13)
     rectangle node{\tiny $\rho$}
     (.15,.17);
  \end{scope}

  \begin{scope}[shift={(-2,{3-0.5*(3+1.8)/3})}]
    \draw[draw=orangeii, fill=orangeii] (0,0) circle (.07);
   \draw[draw=darkblue, thick, fill=white]
     (-.15,-.13)
     rectangle node{\tiny $\rho$}
     (.15,.17);
  \end{scope}

  \begin{scope}[shift={(-2,{3-1.5*(3+1.8)/3})}]
    \draw[draw=orangeii, fill=orangeii] (0,0) circle (.07);
   \draw[draw=darkblue, thick, fill=white]
     (-.15,-.13)
     rectangle node{\tiny $\rho$}
     (.15,.17);
  \end{scope}

  \begin{scope}[shift={(-0,{3-1.5*(3+1.8)/3})}]
    \draw[draw=orangeii, fill=orangeii] (0,0) circle (.07);
   \draw[draw=darkblue, thick, fill=white]
     (-.15,-.13)
     rectangle node{\tiny $\rho$}
     (.15,.17);
  \end{scope}

  \begin{scope}[shift={(-4,{3-2.5*(3+1.8)/3})}]
    \draw[draw=orangeii, fill=orangeii] (0,0) circle (.07);
   \draw[draw=darkblue, thick, fill=white]
     (-.15,-.13)
     rectangle node{\tiny $\rho$}
     (.15,.17);
  \end{scope}

  \begin{scope}[shift={(-0,{3-2.5*(3+1.8)/3})}]
    \draw[draw=orangeii, fill=orangeii] (0,0) circle (.07);
   \draw[draw=darkblue, thick, fill=white]
     (-.15,-.13)
     rectangle node{\tiny $\rho$}
     (.15,.17);
  \end{scope}

 \draw (1.8,3.7) node
   {
     \tiny
     \color{darkblue}
     \begin{tabular}{c}
       close
     \end{tabular}
   };

 \draw[gray, dashed,  thick] (-4.5,3) to (.15+.5,3);
 \draw[gray, dashed, thick] (-4.5,-1.8) to (.15+.5,-1.8);

 \begin{scope}[shift={(.15,.44)}]
   \draw[gray, dashed,  thick] (-4.5,3) to (.15+.5,3);
 \end{scope}
 \begin{scope}[shift={(.15,-3-1.8-.44)}]
   \draw[gray, dashed,  thick] (-4.5,3) to (.15+.5,3);
 \end{scope}

 \begin{scope}[shift={(.33,.84)}]
   \draw[gray, dashed,  thick] (-4.5,3) to (.15+.5,3);
 \end{scope}
 \begin{scope}[shift={(.45,-3-1.8-.755)}]
   \draw[gray, dashed,  thick] (-4.5,3) to (.15+.5,3);
 \end{scope}

\end{tikzpicture}
  }
  \\
  $\;\;\;\;=$
  \raisebox{-90pt}{

\begin{tikzpicture}

 \draw[black, thick, dashed] (.5+.67,.65) ellipse (.8 and 3.2);

 \draw[draw=white, fill=white] (-4.5,3) rectangle (.15+.5,-1.8);

 \draw[thick, white] (-2-.055,-1) to (-2+.055,-1);

 \draw (-4-.2, {3-(3+1.8)/3}) node {\scalebox{.6}{$\mathbf{n}$}};
 \draw (-2-.2, {3-(3+1.8)/3}) node {\scalebox{.6}{$\mathbf{n}$}};
 \draw (-0-.2, {3-(3+1.8)/3}) node {\scalebox{.6}{$\mathbf{n}$}};

 \draw[darkblue, ultra thick]
   (-4,3) .. controls (-4,3-1.2) and (-2,{3-(3+1.8)/3+1.2}) .. (-2,{3-(3+1.8)/3});
 \draw[white, line width=3pt]
   (-2,3) .. controls (-2,3-1.2) and (-4,{3-(3+1.8)/3+1.2}) .. (-4,{3-(3+1.8)/3});
 \draw[darkblue, ultra thick]
   (-2,3) .. controls (-2,3-1.2) and (-4,{3-(3+1.8)/3+1.2}) .. (-4,{3-(3+1.8)/3});

 \draw[darkblue, ultra thick]
   (0,3) to (0,{3-(3+1.8)/3});

 \begin{scope}[shift={(2,{-(3+1.8)/3})}]
 \draw[darkblue, ultra thick]
   (-4,3) .. controls (-4,3-1.2) and (-2,{3-(3+1.8)/3+1.2}) .. (-2,{3-(3+1.8)/3});
 \draw[white, line width=3pt]
   (-2,3) .. controls (-2,3-1.2) and (-4,{3-(3+1.8)/3+1.2}) .. (-4,{3-(3+1.8)/3});
 \draw[darkblue, ultra thick]
   (-2,3) .. controls (-2,3-1.2) and (-4,{3-(3+1.8)/3+1.2}) .. (-4,{3-(3+1.8)/3});
 \end{scope}

 \draw[darkblue, ultra thick]
   (-4,{3-1*(3+1.8)/3}) to (-4,{3-2*(3+1.8)/3});

 \draw[darkblue, ultra thick]
   (-2,{3-2*(3+1.8)/3}) to (-2,{3-3*(3+1.8)/3});

 \draw[white, line width=3pt]
   (-4,{3-2*(3+1.8)/3})
     .. controls
     (-4,{3-2*(3+1.8)/3-1.2+.2})
     and
     (0,{3-3*(3+1.8)/3+1.2+.2})
     ..
     (0,{3-3*(3+1.8)/3});

 \draw[darkblue, ultra thick]
   (-4,{3-2*(3+1.8)/3})
     .. controls
     (-4,{3-2*(3+1.8)/3-1.2+.2})
     and
     (0,{3-3*(3+1.8)/3+1.2+.2})
     ..
     (0,{3-3*(3+1.8)/3});

 \draw[white, line width=3pt]
   (0,{3-2*(3+1.8)/3})
     .. controls
     (0,{3-2*(3+1.8)/3-1.2-.2})
     and
     (-4,{3-3*(3+1.8)/3+1.2-.2})
     ..
     (-4,{3-3*(3+1.8)/3});

 \draw[darkblue, ultra thick]
   (0,{3-2*(3+1.8)/3})
     .. controls
     (0,{3-2*(3+1.8)/3-1.2-.2})
     and
     (-4,{3-3*(3+1.8)/3+1.2-.2})
     ..
     (-4,{3-3*(3+1.8)/3});

 \draw (1.8,3.7) node
   {
     \tiny
     \color{darkblue}
     \begin{tabular}{c}
       close
     \end{tabular}
   };

 \draw[gray, dashed,  thick] (-4.5,3) to (.15+.5,3);
 \draw[gray, dashed, thick] (-4.5,-1.8) to (.15+.5,-1.8);

 \begin{scope}[shift={(.15,.44)}]
   \draw[gray, dashed,  thick] (-4.5,3) to (.15+.5,3);
 \end{scope}
 \begin{scope}[shift={(.15,-3-1.8-.44)}]
   \draw[gray, dashed,  thick] (-4.5,3) to (.15+.5,3);
 \end{scope}

 \begin{scope}[shift={(.33,.84)}]
   \draw[gray, dashed,  thick] (-4.5,3) to (.15+.5,3);
 \end{scope}
 \begin{scope}[shift={(.45,-3-1.8-.755)}]
   \draw[gray, dashed,  thick] (-4.5,3) to (.15+.5,3);
 \end{scope}

\end{tikzpicture}
  }
  & $=$
  &
  \hspace{-.4cm}
  \raisebox{-30pt}{
\begin{tikzpicture}

 \draw[darkblue, ultra thick] (0,0) circle (1);
 \draw[darkblue, ultra thick] (2.5,0) circle (1);

 \draw (130:1.2) node {\scalebox{.6}{$\mathbf{n}$}};

 \begin{scope}[shift={(2.5,0)}]
   \draw (130:1.2) node {\scalebox{.6}{$\mathbf{n}$}};
 \end{scope}

\end{tikzpicture}
  }
  $\;\;\; = \;\;\; n^2$
  \end{tabular}
  }
\end{equation}

\medskip

%%%%%%%%%%%%%%%%%%%%%%%%%%%%%%%%%%%%%%%%%%%%%%%%%%%%%%%%%%%%%%%
\subsection{Quantum states on quantum observable algebras}
\label{QuantumStatesOnQuantumObservableAlgebras}
%%%%%%%%%%%%%%%%%%%%%%%%%%%%%%%%%%%%%%%%%%%%%%%%%%%%%%%%%%%%%%%

\noindent {\bf Quantum observables.}
The following Definition \ref{StarAlgebra}
is often considered for Banach algebras, where it
yields the concept of {\it $C^*$-algebras} (e.g. \cite[\S A4]{Meyer95}\cite[Def. C.1]{Landsman17}).
We need the simple specialization to finite-dimensional star-algebras (e.g. \cite[\S II]{BGQR13}),
or rather the evident mild generalization of that to degreewise
finite-dimensional graded algebras:
\begin{defn}[Star-algebra]
  \label{StarAlgebra}
  A \emph{star-algebra}, for the present purpose, is a
  degreewise finite-dimensional $\mathbb{Z}$-graded associative algebra
  $\mathcal{A}$ over the complex numbers, equipped with an
  involutive anti-linear anti-homomorphism $(-)^\ast$
  (the \emph{star-operation}), hence with a
  function
  \vspace{-2mm}
  $$
    \xymatrix{
      \mathcal{A}
      \ar[rr]^-{(-)^\ast}
      &&
      \mathcal{A}
    }
  $$

    \vspace{-2mm}
\noindent   which satisfies:\\

  \begin{tabular}{llll}
  {\bf (0)} & (degree):
    & $\mathrm{deg}(A) = \mathrm{deg}(A^\ast)$
    &
    \multirow{1}{*}{
      for all homogeneous $A \in \mathcal{A}$
    }
  \\
  {\bf (1)} &  (anti-linearity):
    &
    $\big( a_1 A_1 + a_2 A_2\big)^\ast = \bar a_1 A_1^\ast + \bar a_2 A_2^\ast $
    &
    \multirow{3}{*}{
      for all $a_i \in \mathbb{C}$, $A_i \in \mathcal{A}$
    }
  \\
  {\bf (2)} & (anti-homomorphism):
    &
    $\big( A_1 A_2\big)^\ast = A_2^\ast A_1^\ast$
  \\
  {\bf (3)} & (involution):
    & $\left( (A)^\ast \right)^\ast = A$,
  \end{tabular}\\

  \noindent where $\bar a_i$ denotes the complex conjugate of
  $a_i$.
\end{defn}

We highlight the following example:
\begin{prop}[Star-structure on horizontal chord diagrams]
  \label{StarStructureOnHorizontalChordDiagrams}
  The algebra of horizontal chord diagrams \eqref{AlgebraOfHorizontalChordDiagrams}
  becomes a complex star-algebra (Def. \ref{StarAlgebra})
  via the star-operation
  \begin{equation}
    \label{StarOperationOnLinearCombinationsOfHorizontalChordDiagrams}
    \xymatrix@R=0pt{
      \mathcal{A}^{{}^{\mathrm{pb}}}_{N}
      \ar[rr]^-{ (-)^\ast}
      &&
      \mathcal{A}^{{}^{\mathrm{pb}}}_{N}
      \\
      a_1 \cdot D_1
      +
      a_2 \cdot D_2
      \ar@{}[rr]|-{\longmapsto}
      &&
      \bar a_1 \cdot D_1^\ast
      +
      \bar a_2 \cdot D_2^\ast
      \,,
    }
  \end{equation}
  where
  \begin{equation}
    \label{StarOperationOnHorizontalChordDiagrams}
    \xymatrix{
      \mathcal{D}^{{}^{\mathrm{pb}}}_{\!N}
      \ar[rr]^-{ (-)^\ast}
      &&
      \mathcal{D}^{{}^{\mathrm{pb}}}_{\!N}
    }
  \end{equation}
  is the operation that
  reverses the orientation of strands in a chord diagram \eqref{SetOfHorizontalChordDiagrams},
  hence which reverses the ordering of the corresponding lists
  \eqref{HorizontalChordDiagramAsListOfChords}.
\end{prop}
For example:
$$
\left(
a
\cdot
\left[
\scalebox{.8}{
\raisebox{-105pt}{
\begin{tikzpicture}
 \draw[thick, orangeii] (0,0) to node{$\phantom{AA}$} (1,0);
 \begin{scope}
   \clip (0,-.1) rectangle (.1,+.1);
   \draw[draw=orangeii, fill=orangeii] (0,0) circle (.07);
 \end{scope}
 \begin{scope}[shift={(1,0)}]
   \clip (-.1,-.1) rectangle (0,+.1);
   \draw[draw=orangeii, fill=orangeii] (0,0) circle (.07);
 \end{scope}
 \draw[thick, orangeii] (1,1) to node{$\phantom{AA}$} (2,1);
 \begin{scope}[shift={(1,1)}]
   \clip (0,-.1) rectangle (.1,+.1);
   \draw[draw=orangeii, fill=orangeii] (0,0) circle (.07);
 \end{scope}
 \begin{scope}[shift={(2,1)}]
   \clip (-.1,-.1) rectangle (0,+.1);
   \draw[draw=orangeii, fill=orangeii] (0,0) circle (.07);
 \end{scope}
 \draw[thick, orangeii] (2,-3) to (3,-3);
 \begin{scope}[shift={(2,-3)}]
   \clip (0,-.1) rectangle (.1,+.1);
   \draw[draw=orangeii, fill=orangeii] (0,0) circle (.07);
 \end{scope}
 \begin{scope}[shift={(3,-3)}]
   \clip (-.1,-.1) rectangle (0,+.1);
   \draw[draw=orangeii, fill=orangeii] (0,0) circle (.07);
 \end{scope}
 \draw[thick, orangeii]
   (1,-1)
   to
   node{\colorbox{white}{\hspace{-2pt}}}
   (3,-1);
 \begin{scope}[shift={(1,-1)}]
   \clip (0,-.1) rectangle (.1,+.1);
   \draw[draw=orangeii, fill=orangeii] (0,0) circle (.07);
 \end{scope}
 \begin{scope}[shift={(3,-1)}]
   \clip (-.1,-.1) rectangle (0,+.1);
   \draw[draw=orangeii, fill=orangeii] (0,0) circle (.07);
 \end{scope}
 \draw[thick, orangeii]
   (1,-2) to (4,-2);
 \draw (2,-2) node {\colorbox{white}{\hspace{-2pt}}};
 \draw (3,-2) node {\colorbox{white}{\hspace{-2pt}}};
 \begin{scope}[shift={(1,-2)}]
   \clip (0,-.1) rectangle (.1,+.1);
   \draw[draw=orangeii, fill=orangeii] (0,0) circle (.07);
 \end{scope}
 \begin{scope}[shift={(4,-2)}]
   \clip (-.1,-.1) rectangle (0,+.1);
   \draw[draw=orangeii, fill=orangeii] (0,0) circle (.07);
 \end{scope}
 \draw[thick, orangeii]
   (0,2) to (3,2);
 \draw (1,2) node {\colorbox{white}{\hspace{-2pt}}};
 \draw (2,2) node {\colorbox{white}{\hspace{-2pt}}};
 \begin{scope}[shift={(0,2)}]
   \clip (0,-.1) rectangle (.1,+.1);
   \draw[draw=orangeii, fill=orangeii] (0,0) circle (.07);
 \end{scope}
 \begin{scope}[shift={(3,2)}]
   \clip (-.1,-.1) rectangle (0,+.1);
   \draw[draw=orangeii, fill=orangeii] (0,0) circle (.07);
 \end{scope}
 \draw[ultra thick, darkblue] (0,3) to  (0,-4);
 \draw[ultra thick, darkblue] (1,3) to  (1,-4);
 \draw[ultra thick, darkblue] (2,3) to  (2,-4);
 \draw[ultra thick, darkblue] (3,3) to  (3,-4);
 \draw[ultra thick, darkblue] (4,3) to  (4,-4);
 \draw[->] (4+.3,-3) to
   (4+.3,-4);
%
% \draw (-.7,-4.4) node {\small $a =$};
 \draw (0,-4.3) node {\scalebox{.7}{$1$}};
 \draw (1,-4.3) node {\scalebox{.7}{$2$}};
 \draw (2,-4.3) node {\scalebox{.7}{$3$}};
 \draw (3,-4.3) node {\scalebox{.7}{$4$}};
 \draw (4,-4.3) node {\scalebox{.7}{$5$}};
\end{tikzpicture}
}
}
\right]
\right)^\ast
\;\;=\;\;
\bar a
\cdot
\left[
\scalebox{.8}{
\raisebox{-105pt}{
\begin{tikzpicture}
 \begin{scope}[shift={(0,-1)}]
 \begin{scope}[yscale=-1]
 \draw[thick, orangeii] (0,0) to node{$\phantom{AA}$} (1,0);
 \begin{scope}
   \clip (0,-.1) rectangle (.1,+.1);
   \draw[draw=orangeii, fill=orangeii] (0,0) circle (.07);
 \end{scope}
 \begin{scope}[shift={(1,0)}]
   \clip (-.1,-.1) rectangle (0,+.1);
   \draw[draw=orangeii, fill=orangeii] (0,0) circle (.07);
 \end{scope}
 \draw[thick, orangeii] (1,1) to node{$\phantom{AA}$} (2,1);
 \begin{scope}[shift={(1,1)}]
   \clip (0,-.1) rectangle (.1,+.1);
   \draw[draw=orangeii, fill=orangeii] (0,0) circle (.07);
 \end{scope}
 \begin{scope}[shift={(2,1)}]
   \clip (-.1,-.1) rectangle (0,+.1);
   \draw[draw=orangeii, fill=orangeii] (0,0) circle (.07);
 \end{scope}
 \draw[thick, orangeii] (2,-3) to (3,-3);
 \begin{scope}[shift={(2,-3)}]
   \clip (0,-.1) rectangle (.1,+.1);
   \draw[draw=orangeii, fill=orangeii] (0,0) circle (.07);
 \end{scope}
 \begin{scope}[shift={(3,-3)}]
   \clip (-.1,-.1) rectangle (0,+.1);
   \draw[draw=orangeii, fill=orangeii] (0,0) circle (.07);
 \end{scope}
 \draw[thick, orangeii]
   (1,-1)
   to
   node{\colorbox{white}{\hspace{-2pt}}}
   (3,-1);
 \begin{scope}[shift={(1,-1)}]
   \clip (0,-.1) rectangle (.1,+.1);
   \draw[draw=orangeii, fill=orangeii] (0,0) circle (.07);
 \end{scope}
 \begin{scope}[shift={(3,-1)}]
   \clip (-.1,-.1) rectangle (0,+.1);
   \draw[draw=orangeii, fill=orangeii] (0,0) circle (.07);
 \end{scope}
 \draw[thick, orangeii]
   (1,-2) to (4,-2);
 \draw (2,-2) node {\colorbox{white}{\hspace{-2pt}}};
 \draw (3,-2) node {\colorbox{white}{\hspace{-2pt}}};
 \begin{scope}[shift={(1,-2)}]
   \clip (0,-.1) rectangle (.1,+.1);
   \draw[draw=orangeii, fill=orangeii] (0,0) circle (.07);
 \end{scope}
 \begin{scope}[shift={(4,-2)}]
   \clip (-.1,-.1) rectangle (0,+.1);
   \draw[draw=orangeii, fill=orangeii] (0,0) circle (.07);
 \end{scope}
 \draw[thick, orangeii]
   (0,2) to (3,2);
 \draw (1,2) node {\colorbox{white}{\hspace{-2pt}}};
 \draw (2,2) node {\colorbox{white}{\hspace{-2pt}}};
 \begin{scope}[shift={(0,2)}]
   \clip (0,-.1) rectangle (.1,+.1);
   \draw[draw=orangeii, fill=orangeii] (0,0) circle (.07);
 \end{scope}
 \begin{scope}[shift={(3,2)}]
   \clip (-.1,-.1) rectangle (0,+.1);
   \draw[draw=orangeii, fill=orangeii] (0,0) circle (.07);
 \end{scope}
 \draw[ultra thick, darkblue] (0,3) to  (0,-4);
 \draw[ultra thick, darkblue] (1,3) to  (1,-4);
 \draw[ultra thick, darkblue] (2,3) to  (2,-4);
 \draw[ultra thick, darkblue] (3,3) to  (3,-4);
 \draw[ultra thick, darkblue] (4,3) to  (4,-4);
 \end{scope}
 \end{scope}
 \draw[->] (4+.3,-3) to
   (4+.3,-4);
%
% \draw (-.7,-4.4) node {\small $a =$};
 \draw (0,-4.3) node {\scalebox{.7}{$1$}};
 \draw (1,-4.3) node {\scalebox{.7}{$2$}};
 \draw (2,-4.3) node {\scalebox{.7}{$3$}};
 \draw (3,-4.3) node {\scalebox{.7}{$4$}};
 \draw (4,-4.3) node {\scalebox{.7}{$5$}};
\end{tikzpicture}
}
}
\right]
$$
\begin{proof}
  The statement evidently holds before quotienting out the
  2T-relations \eqref{2TRelationsOnHorizontalChordDiagrams}
  and 4T-relations \eqref{4TRelationOnHorizontalChordDiagrams}
  in
  \eqref{AlgebraOfHorizontalChordDiagrams}; and the reversal operation manifestly
  preserves these relations, hence preserves the ideals they generate, hence passes to the
  quotient.

  More abstractly, this star-involution is the involutory antipode of the
  Hopf algebra structure on the homology of loop spaces
  (\cite[p. 262]{MilnorMoore65})
  under the identification of
  horizontal chord diagrams with the homology of the loop space of
  configuration spaces of points (\cite[Thm. 4.1]{Kohno02});
  see around \eqref{FromCohomotopyToChordDiagramObservables} in
  \cref{PhysicsInterpretation} below for more on this perspective.
\end{proof}

\begin{prop}[Reversed chord diagrams give inverse permutations]
  \label{AssigningPermutationsToChordDiagramsIsStarHomomorphism}
  The function \eqref{MonoidHomomorphismFromHorizontalChordsToPermutations}
  that sends horizontal chord diagrams to permutations
  sends reversed chord diagrams \eqref{StarOperationOnHorizontalChordDiagrams}
  to inverse permutations:
  \begin{equation}
    \begin{tikzcd}[row sep=small]
      \mathcal{D}^{{}^{\mathrm{pb}}}_{\!N}
      \ar[
        rr,
        "\mathrm{perm}"
      ]
      \ar[
        d,
        "(-)^\ast"{left}
      ]
      &&
      \mathrm{Sym}
      \big(
        N
      \big)
      \ar[
        d,
        "(-)^{-1}"
      ]
      \\
      \mathcal{D}^{{}^{\mathrm{pb}}}_{\!N}
      \ar[
        rr,
        "\mathrm{perm}"
      ]
      &&
      \mathrm{Sym}
      \big(
        N
      \big)
    \end{tikzcd}
    {\phantom{AAA}}
    {\phantom{AAA}}
    \mathrm{perm}(D^\ast)
    \;=\;
    \big( \mathrm{perm}(D)  \big)^{-1}
  \end{equation}
\end{prop}
\begin{proof}
  This follows immediately from the definition
  \eqref{HorizontalChordDiagramAsListOfChords}
  and the fact that
  any transposition is its own inverse.
\end{proof}

\begin{example}[Perm is a star-monoid homomorphism]
  \label{PermSendsDStarDToPermDInversePermD}
  Since $\mathrm{perm}$ is a monoid homomorphism by
  construction \eqref{MonoidHomomorphismFromHorizontalChordsToPermutations}
  in Def. \ref{AlgebraOfAndWeightSystemsOnHorizontalChordDiagrams},
  Prop. \ref{AssigningPermutationsToChordDiagramsIsStarHomomorphism}
  may be read as saying that it is in fact a
  {\it star-monoid homomorphism}. In particular, we have:
  $$
    \begin{aligned}
      \mathrm{perm}
      \big(
        D_1^\ast \cdot D_2
      \big)
      \;=\;
      \mathrm{perm}(D_1)^{-1}
      \circ
      \mathrm{perm}(D_2)
      \,.
    \end{aligned}
  $$
\end{example}

\medskip
\noindent {\bf Quantum states.}
The following is the standard mathematical formulation of what are
often called  {\it mixed states} or {\it density matrices} in quantum physics,
subsuming, as a special case, the traditional pure states that may be identified
with elements of a Hilbert space.
\begin{defn}[State on a star-algebra. e.g. {\cite[\S I.1.1]{Meyer95}\cite[Def. 2.4]{Landsman17}}]
  \label{States}
  Given a star-algebra $(\mathcal{A}, (-)^\ast)$
  (Def. \ref{StarAlgebra}), a \emph{state} is a complex-linear function
  \vspace{-2mm}
  \begin{equation}
    \label{LinearFormOnAlgebra}
    \rho
    \;:\;
    \mathcal{A}
    \longrightarrow
    \mathbb{C}
  \end{equation}

  \vspace{-2mm}
\noindent  which satisfies:\\

  \vspace{-2mm}
  \begin{tabular}{llll}
    {\bf (1)} & (positivity):
      & $ \rho\big( A^\ast A \big) \geq 0 \in \mathbb{R} \subset \mathbb{C}$
      & for all $A \in \mathcal{A}$;
    \\
    {\bf (2)} & (normalization):
      & $\rho(\mathbf{1}) = 1$
      & for $\mathbf{1} \in \mathcal{A}$ the algebra unit.
  \end{tabular}\\

\end{defn}

\begin{remark}[S-Matrices -- part of the {\it GNS construction}, e.g. {\cite[Prop. 4.5.1 \& p. 270]{KadisonRingrose97}\cite[(II.5)]{BGQR13}}]
  \label{QuantumStatesAsPositiveSemiDefiniteBilinearForms}
  Given a star-algebra $(\mathcal{A}, (-)^\ast)$ (Def. \ref{StarAlgebra}),
  we may identify any linear form $\rho$ \eqref{LinearFormOnAlgebra}
  on the underlying vector space $\mathcal{A}$ with the following
  sesquilinear form
  \vspace{-2mm}
  \begin{equation}
    \label{InducedBilinearFormFromLinearFormOnStarAlgebra}
    \begin{tikzcd}[row sep=-2pt]
      \mathcal{A} \otimes \mathcal{A}
      \ar[
        rr,
        "\rho\left( (-)^\ast \cdot (-) \right)"
      ]
      &&
      \mathcal{C}
      \\
      (A_1, A_2)
      &\longmapsto&
      \rho
      \big(
        A_1^\ast \cdot A_2
      \big)
      \,.
    \end{tikzcd}
  \end{equation}

  \vspace{-2mm}
\noindent
  Observe that a linear form is a state (Def. \ref{States})
  precisely if its induced sesquilinear form 
  \eqref{InducedBilinearFormFromLinearFormOnStarAlgebra}
  is
  (normalized to $\rho(\mathbf{1}^\ast \cdot \mathbf{1}) = 1$ and)
  positive (semi-)definite:
  $$
    \mbox{$\rho(-)$ is a state}
    \;\;\;\;\;\Leftrightarrow\;\;\;\;\;
    \mbox{
      $
        \rho
        \big(
          (-)^\ast \cdot (-)
        \big)
      $
      is normalized and
      positive (semi-)definite.
    }
  $$
\end{remark}

\begin{remark}[Positivity]
  The point of Def. \ref{States} is the positivity condition
  (which might rather deserve to be called \emph{semi-positivity},
  by Rem. \ref{QuantumStatesAsPositiveSemiDefiniteBilinearForms};
  but \emph{positivity} is the established terminology here)
  while the normalization condition
  is just that: If $\rho$ is a (semi-)positive linear map
  with $\rho \neq 0$ then
  $\frac{1}{\rho(\mathbf{1})} \cdot \rho$ is a state.
\end{remark}

\begin{defn}[Weight systems that are quantum states]
  \label{NotionOfWeightSystemsThatAreQuantumStates}
  Here we say that a weight system on horizontal chord diagrams
  (Def. \ref{AlgebraOfAndWeightSystemsOnHorizontalChordDiagrams})
  {\it is a quantum state} if it is a state (Def. \ref{States})
  with respect to the canonical star-algebra structure on horizontal chord
  diagrams from Prop. \ref{StarStructureOnHorizontalChordDiagrams}.
\end{defn}

We will show that the fundamental weight systems
(Def. \ref{FundamentalWeightSystemOfgl2}) are quantum states (Def. \ref{NotionOfWeightSystemsThatAreQuantumStates})
by regarding them as kernels in geometric group theory, in Prop. \ref{FundamentalWeightSystemExpressedViaCayleyGraphDistance} below.
Notice that from any set of quantum states like this, we obtain
at least a convex hull of all their operator images as further states:

\begin{example}[Convex combinations of quantum states]
  \label{Mixture}
  For $k \in \mathbb{N}_+$
  the {\it mixture}
  of a $k$-tuple
  $\big(
    \rho_i : \mathcal{A} \to \mathbb{C}\big)_{1 \leq i \leq k}$
  of quantum states
  (Def. \ref{NotionOfWeightSystemsThatAreQuantumStates})
  for {\it probability distribution}
  $\big(p_i \in \mathbb{R}_{\geq 0}\big)_{1 \leq i \leq k}$,
  $\underset{i}{\sum}p_i = 1$,
  is the quantum state given by the convex linear combination
  $
    \underset{i}{\sum} p_i \cdot \rho_i
    \;\;\;\in\;
    \mathcal{A}^{\ast}
    \,.
  $
\end{example}
\begin{example}[Operator-state correspondence]
  \label{OperatorStateCorrespondence}
  For $\rho : \mathcal{A} \to \mathbb{C}$ any quantum state (Def. \ref{States}),
  every non-null observable $O \in \mathcal{A}$, $\rho(O^\ast O) \neq 0$
  induces another state $\rho_O$ given by
  $
    \rho_O
    \big(
      A
    \big)
    \;\coloneqq\;
    \tfrac{1}{\rho(O^\ast O)}
    \cdot
    \rho
    \big(
      O^\ast \cdot A \cdot O
    \big)
    \,.
  $
\end{example}

%%%%%%%%%%%%%%%%%%%%%%%%%%%%%%%%%%%%%%%%%%%%%%%%%%%%%
\subsection{Cayley distance kernels on symmetric groups}
\label{GeometricGroupTheory}
%%%%%%%%%%%%%%%%%%%%%%%%%%%%%%%%%%%%%%%%%%%%%%%%%%%%%

\noindent {\bf Cayley distance metric on symmetric groups.}
The following is at the heart of geometric group theory (e.g. \cite{DrutuKapovich18}).

\begin{defn}[Cayley distance (e.g. {\cite[p. 112]{Diaconis88}})]
  \label{CayleyDistance}
  The \emph{Cayley graph} of the symmetric group $\mathrm{Sym}(N)$
  is the undirected graph whose vertices are
  the group elements and which has exactly one edge between
  any pair of group elements if they differ by
  composition with a single transposition
  \eqref{ATransposition} on the right:
  $$
    \begin{tikzcd}[row sep=0pt, column sep=tiny]
      {}
      \ar[dr,-,dotted, very thick]
      & && &
      {}
      \\
      {}
      \ar[r,-,dotted,  very thick]
      &
      \sigma
      \ar[
       rr,-
      ]
      &
      {\phantom{A}}
      &
      \sigma \circ t_{i j}
      \ar[r,-,dotted,  very thick]
      \ar[ur,-,dotted,  very thick]
      \ar[dr,-,dotted,  very thick]
      &
      {}
      \\
      {}
      \ar[ur,-,dotted,  very thick]
      & && &
      {}
    \end{tikzcd}
  $$
 We denote  the corresponding \emph{Cayley graph distance function} by
  $$
    d_C
    \;:\;
    \begin{tikzcd}
      \mathrm{Sym}(N)
      \times
      \mathrm{Sym}(N)
      \ar[r]
      &
      \mathbb{N}
      \,,
    \end{tikzcd}
  $$
  hence:
  \begin{equation}
    \label{CayleyDistanceFunction}
    d_C
    (
      \sigma_1,\,
      \sigma_2
    )
    \;=\;
    d_C
    (
      e,\,
      \sigma_1^{-1} \circ \sigma_2
    )
    \;=\;
    \left\{
    \mbox{
      \begin{tabular}{l}
        minimal number $k$
                of transpositions
        \\
        $\big\{
          t_{i_1 j_1}, \cdots, t_{i_k j_k} \in \mathrm{Sym}(N)
        \big\}$
        \\
        such that
        \\
        $\sigma_1^{-1} \circ \sigma_2
        \,=\, t_{i_1 j_1} \circ \cdots \circ t_{i_k j_k}$\;.
      \end{tabular}
    }
    \right.
  \end{equation}
\end{defn}

\begin{example}[Cayley graph of $\mathrm{Sym}(3)$]
  \label{CayleyGraphOfSym3}
  The Cayley graph of the symmetric group on $N = 3$ elements, with
  edges for arbitrary transpositions, looks as follows:
\begin{equation}
\label{CayleyGraphForSym3}
\adjustbox{scale=.7}{
\begin{tikzcd}[row sep=-6pt, column sep=14pt]
  123
  \ar[rrrr,-]
  \ar[ddddd,-]
  \ar[ddrrr,-]
  &&
  &&
  132
  \ar[ddddd,-]
  \\
  {\phantom{A} \atop {\phantom{A} \atop {\phantom{A} \atop {\phantom{A} \atop {\phantom{A} \atop {\phantom{A} \atop {\phantom{A}}}}}}}}
  \\
  &&&
  \scalebox{.8}{$321$}
  \ar[dddr,-]
  \\
  &
  \scalebox{1.2}{231}
  \ar[ddl,-]
  \ar[uuurrr,-, crossing over]
  \ar[urr,-]
  \\
  {\phantom{A} \atop {\phantom{A} \atop {\phantom{A} \atop {\phantom{A} \atop {\phantom{A} \atop {\phantom{A} \atop {\phantom{A}}}}}}}}
  \\
  213
  \ar[rrrr,-]
  &&&&
  312
\end{tikzcd}
}
\end{equation}
Here, e.g., ``$231$'' is shorthand for the permutation
$\sigma \in \mathrm{Sym}(3)$
with  $\sigma(1) = 2$, $\sigma(2) =3$, $\sigma(3) = 1$.
If we order these 6 permutations to a linear basis for
$\mathbb{C}[\mathrm{Sym}(N)]$ as follows
$$
  \big[
    123
    ,\,
    213
    ,\,
    132
    ,\,
    321
    ,\,
    312
    ,\,
    231
  \big],
$$
then the matrix of Cayley distances \eqref{CayleyDistanceFunction}
between these is
\begin{equation}
  \label{MatrixOfCayleyDistancesOnSym3}
  \big[d_c\big]
  \;\;=\;\;
  \left[
  \begin{array}{cccccc}
    0 & 1 & 1 & 1  & 2 &  2
    \\
    1 & 0 & 2 & 2 & 1 & 1
    \\
    1 & 2 & 0 & 2 & 1 & 1
    \\
    1 & 2 & 2 & 0 & 1 & 1
    \\
    2 & 1 & 1 & 1 & 0 & 2
    \\
    2 & 1 & 1 & 1 & 2 & 0
  \end{array}
  \right].
\end{equation}

\end{example}

For later reference, we record the basic properties of the Cayley distance:

\begin{lemma}
  \label{LeftInvarianceOfCayleyDistance}
  The Cayley distance \eqref{CayleyDistanceFunction} is {\it left invariant}:
  $
    \underset{
      \sigma \in \mathrm{Sym}(N)
    }{\forall}
    \;
    d_C\big( \sigma \circ (-),\, \sigma \circ (-)  \big)
    \;=\;
    d_C\big(- ,\, - \big)
    \,.
  $
\end{lemma}

\begin{lemma}[Cayley's formula, e.g. {\cite[p. 118]{Diaconis88}}]
  \label{CayleyDistanceFunctionInTermsOfCycles}
  The Cayley distance function \eqref{CayleyDistanceFunction}
  equals $N$ minus the number of cycles in the
  permutation:
  \begin{equation}
    \label{CayleyFormula}
    d_C(\sigma_1, \sigma_2)
    \;=\;
    N
    -
    \#\mathrm{cycles}(\sigma_1^{-1} \circ \sigma_2)\;.
  \end{equation}
\end{lemma}
\begin{proof}
  Since both sides of the equation are invariant under left multiplication
  (Lemma \ref{LeftInvarianceOfCayleyDistance}),
  it is sufficient to show the statement for $\sigma_1 = e$, hence
  for any $\sigma = \sigma_2$.
  Here, notice first that any cyclic permutation of $k+1$
  elements is the product of no fewer than $k$ transpositions:
  \begin{equation}
    \label{CyclicPermutationAsCompositeOfTranspositions}
    \left(
    \begin{array}{ccccc}
      1 & 2 & 3 & \cdots & k + 1
      \\
      k + 1 & 1 & 2 & \cdots & k
    \end{array}
    \right)
    \;=\;
    \underset{
      \mbox{\tiny $k$ transpositions}
    }{
    \underbrace{
      t_{k, k-1}
        \,\circ\,
      \cdots
        \,\circ\,
      t_{3,2}
        \,\circ\,
      t_{2,1}
        \,\circ\,
      t_{1,k+1}
         }
    }\;,
  \end{equation}
  where we understand that for $k = 0$ the composite on the right is the
  neutral element.
  But every permutation $\sigma$ is the composite of such
  cyclic permutations,
  one for each of its cycles, with those in different cycles commuting with each other.
  Since \eqref{CyclicPermutationAsCompositeOfTranspositions}
  has one transposition fewer than the number of elements, this implies
  that for every cycle the minimum number of transpositions needed
  is reduced by one.
\end{proof}

\begin{lemma}[Cayley distance is preserved by inclusion of symmetric groups]
  \label{CayleyDistanceIsPreservedByInclusionOfSymmetricGroups}
  The canonical inclusion of symmetric groups
  $$
    \begin{tikzcd}
      \mathrm{Sym}(N)
      \ar[
        r,
        hook,
        "i"
      ]
      &
      \mathrm{Sym}(N + 1)
    \end{tikzcd}
  $$
  preserves Cayley distance (Def. \ref{CayleyDistance}):
  $$
    \underset{
      { \sigma_1, \sigma_2 }
      \atop
      {\in \mathrm{Sym}(N)}
    }{\forall}
    \;\;
    d_C( \sigma_1, \sigma_2  )
    \;=\;
    d_C\big( i(\sigma_1), i(\sigma_2)  \big)
    \,.
  $$
  In other words, the Cayley distance matrix
  $[d_C]$ of $\mathrm{Sym}(N)$
  is the principal submatrix of that of $\mathrm{Sym}(N+1)$ on the
  permutations in the image of the inclusion $i$.
\end{lemma}
\begin{proof}
  Observing that
  $
    \#\mathrm{cycles}
    \big(
      i(\sigma_1)^{-1} \circ i(\sigma_2)
    \big)
    =
    \#\mathrm{cycles}
    \big(
      \sigma_1^{-1} \circ \sigma_2
    \big)
    +
    1
    \,,
  $
  the claim follows by Cayley's formula (Lemma \ref{CayleyDistanceFunctionInTermsOfCycles}).
\end{proof}

\noindent
{\bf Cayley distance kernels on symmetric groups.} We now consider the corresponding kernels.

\begin{defn}[Cayley distance kernel]
  \label{CayleyDistanceKernel}
  The
  {\it Cayley distance kernel}
  at
  {\it inverse temperature} $\beta \in \mathbb{R}_{\geq 0}$
  is the function
  on pairs of permutations that is
  given by the exponential of the
  Cayley distance (Def. \ref{CayleyDistance}) weighted by $- \beta$:
  $$
    \begin{tikzcd}[row sep=0pt]
      e^{- \beta \cdot d_C(-,-)}
      \;:\;
      \mathrm{Sym}(N)
      \times
      \mathrm{Sym}(N)
      \ar[r]
      &
      \mathbb{R}
      \,.
    \end{tikzcd}
  $$
  Understood as a matrix, we naturally conflate this with its
  induced sesqui-linear form:
  \begin{equation}
    \label{CayleySesquilinearForm}
    \begin{tikzcd}[row sep=0pt]
      \mathbb{C}[\mathrm{Sym}(N)]
      \otimes
      \mathbb{C}[\mathrm{Sym}(N)]
      \ar[
        rr,
        "{
          \langle -,-\rangle_\beta
        }"
      ]
      &&
      \mathbb{C}
      \\
      \Big(\;
        \underset{
          \sigma_1 \in \mathrm{Sym}(N)
        }{\sum}
        a_{\sigma_1} \cdot \sigma_1
        ,\,
        \underset{
          \sigma_2 \in \mathrm{Sym}(N)
        }{\sum}
        b_{\sigma_2} \cdot \sigma_2
      \Big)
      &\longmapsto&
      {
        \underset{
          {\sigma_1, \sigma_2}
          \atop
          {\in \mathrm{Sym}(N)}
        }{\sum}
        \bar a_{\sigma_1}
          \cdot
             b_{\sigma_2}
          \cdot
        e^{ - \beta \cdot d_C(\sigma_1, \sigma_2) }
        \,.
      }
    \end{tikzcd}
  \end{equation}
%  and with the corresponding quadratic form:
%  \begin{equation}
%    \label{CayleyDistanceKernelAsQuadraticForm}
%    \begin{tikzcd}[row sep=0pt]
%      \mathbb{C}[\mathrm{Sym}(N)]
%      \ar[
%        rr,
%        "{
%          \left\vert - \right\vert^2_\beta
%        }"
%      ]
%      &&
%      \mathbb{R}
%      \\
%      a &\mapsto&
%      \langle a, a \rangle_\beta
%    \end{tikzcd}
%  \end{equation}
\end{defn}

\begin{remark}[Related literature]
\label{ReferencesOnCayleyDistanceKernel}
The Cayley distance kernel (Def. \ref{CayleyDistanceKernel})
is mentioned, for instance, in
\cite[\S 4]{FlignerVerducci86}\cite[\S 4]{DiaconisHanlon92}\cite[p. xx]{FlignerVerducci93},
but has received less attention than
related kernels in geometric group theory.
Notably the closely related {\it Mallows kernel}
(see, e.g., \cite[\S 6B]{Diaconis88}), which is instead formed from the
{\it Kendall distance} $d_K$ given by the minimum number of {\it adjacent}
permutations, is widely studied and has recently been proven \cite{JiaoVert18}
to be positive definite, generally.
In contrast, the Cayley distance kernel may become indefinite for small $\beta$
(Example \ref{CayleyDistanceKernelOnSym3}) and its general dependency on
$\beta$ had previously remained unknown.
\end{remark}

We now relate Cayley distance kernels to the weight systems from \cref{WeightSystemsOnChordDiagrams}:
\begin{prop}[Fundamental $\mathfrak{gl}(n)$-weight system is Cayley distance kernel at
$\beta = \mathrm{ln}(n)$]
  \label{FundamentalWeightSystemExpressedViaCayleyGraphDistance}
  The fundamental $\mathfrak{gl}(n)$-weight system
  $w_{(\mathfrak{gl}(n),\mathbf{n})}$ (Def. \ref{FundamentalWeightSystemOfgl2}),
  regarded as a sesquilinear form \eqref{InducedBilinearFormFromLinearFormOnStarAlgebra}
  on horizontal chord diagrams of $N$ strands,
  equals the Cayley distance kernel (Def. \ref{CayleyDistanceKernel})
  at inverse temperature $\beta = \mathrm{ln}(n)$
  on the corresponding permutations \eqref{MonoidHomomorphismFromHorizontalChordsToPermutations}
  of $N$ elements:

  \begin{equation}
    w_{\mathbf{n}}
    \Big(
      \Big(
        \underset{i}{\sum} a_i [D_i]
      \Big)^\ast
      \cdot
      \Big(
        \underset{j}{\sum} b_j [D_j]
      \Big)
    \Big)
    \;\;=\;\;
    \underset{i,j}{\sum}
    \,
    \bar a_i
         b_j
      \cdot
    e^{
      -
      \mathrm{ln}(n)
        \cdot
      d_C
      \big(
        \mathrm{perm}(D_i),
        \mathrm{perm}(D_j)
      \big)
    }
    \,.
  \end{equation}
\end{prop}
\begin{proof}
We compute as follows:
\begin{equation}
\begin{aligned}
  w_{\mathbf{n}}
  \Big(
    \Big(
      \underset{i}{\sum} a_i [D_i]
    \Big)^\ast
    \cdot
    \Big(
      \underset{j}{\sum} b_j [D_j]
    \Big)
  \Big)
  &
  \;=\;
  \underset{
    i,j
  }{\sum}
  \;
  \bar a_i b_j
  \cdot
  w_{\mathbf{n}}
  \big(
    D_i^\ast \cdot D_j
  \big)
  \\
  &
  \;=\;
  \underset{
    i,j
  }{\sum}
  \;
  \bar a_i b_j
  \cdot
  e^{
    \mathrm{ln}(n)
      \cdot
    \Big(
      \#\mathrm{cycles}
      \big(
        \mathrm{perm}(D_i^\ast \cdot D_j)
      \big)
      -
      N
    \Big)
  }
  \\
  &
  \;=\;
  \underset{
    i,j
  }{\sum}
  \;
  \bar a_i b_j
  \cdot
  e^{
    \mathrm{ln}(n)
      \cdot
    \Big(
      \#\mathrm{cycles}
      \big(
        \mathrm{perm}(D_i^\ast)^{-1}
          \circ
        \mathrm{perm}(D_j)
      \big)
      -
      N
    \Big)
  }
  \\
  &
  \;=\;
  \underset{
    i,j
  }{\sum}
  \;
  \bar a_i b_j
  \cdot
  e^{
    -
    \mathrm{ln}(n)
      \cdot
    d_C
    \big(
      \mathrm{perm}(D_i),
      \mathrm{perm}(D_j)
    \big).
  }
\end{aligned}
\end{equation}
Here the first step is sesqui-linearity,
the second step is Cor. \ref{FundamentalWeightSystemInTermsOfCycles},
the third step is Ex. \ref{PermSendsDStarDToPermDInversePermD},
and the last step is Lemma \ref{CayleyDistanceFunctionInTermsOfCycles}.
\end{proof}

\medskip

It follows immediately
that the fundamental $\mathfrak{gl}(n)$-weight system
on $\mathcal{A}^{{}^{\mathrm{pb}}}_N$
is a quantum state
precisely if the Cayley distance kernel on $\mathrm{Sym}(N)$
is positive (semi-)definite at
$\beta = \mathrm{ln}(n)$:
\begin{equation}
  \label{FundamentalWeightSystemIsQuantumStateIffCayleyDistanceKernelIsPositiveSemiDefinite}
  \mbox{
    $w_{(\mathfrak{gl}(n), \mathbf{n})}$
    is a quantum state
  }
  {\phantom{AAA}}
  \Leftrightarrow
  {\phantom{AAA}}
  \mbox{
    $e^{- \mathrm{ln}(n) \cdot d_C}$
    is positive (semi-)definite
  }
  \,.
\end{equation}
\noindent
Therefore we now turn to analyzing the positivity of the Cayley distance kernel.

\medskip

%%%%%%%%%%%%%%%%%%%%%%%%%%%%%%%%%%%%%%%%%%%%%%%%%%%%%
\section{Positivity of the Cayley distance kernel}
%%%%%%%%%%%%%%%%%%%%%%%%%%%%%%%%%%%%%%%%%%%%%%%%%%%%%

We discuss the (non-/semi-)positivity of the Cayley distance kernel
(hence of its lowest eigenvalue)
in dependence of the inverse temperature parameter $\beta$.

\noindent Throughout, we take {\it semi-definite} to imply that there is {\it at least one}
vanishing eigenvalue.

\medskip

To start with, it is instructive to look at the first non-trivial case:

\begin{example}[Cayley distance kernel on $\mathrm{Sym}(3)$]
  \label{CayleyDistanceKernelOnSym3}
  In the case $N = 3$,
  with the matrix of Cayley distances given by \eqref{MatrixOfCayleyDistancesOnSym3}
  in Example \ref{CayleyGraphOfSym3},
  the eigenvalues of
  the corresponding matrix
  $\big[  e^{- \beta \cdot d_c} \big] $
  representing the
  Cayley distance kernel
  \eqref{CayleySesquilinearForm}
  are readily computed to be
  \vspace{-2mm}
  $$
    \frac{
      e^{2 \beta} \pm 3 e^\beta + 2
    }{
      e^{2 \beta}
    }
    \;\;\;
    \mbox{and}
    \;\;\;
    \frac{
      e^{2 \beta} - 1
    }{
      e^{2 \beta}
    }
    \,,
  $$
  where the first two have multiplicity 1, while the last has multiplicity 4.
  For the given domain of the parameter $\beta \in \mathbb{R}_{\geq 0}$
  all these eigenvalues are always positive, except for one which may
  change sign as $\beta$ varies:
  $$
    e^{2 \beta} - 3 e^\beta + 2
    \;\;\mbox{is}\;\;
    \left\{
    \def\arraystretch{1.1}
    \begin{array}{ccl}
      = 0 &\mbox{for}& \beta = \mathrm{ln}(1),
      \\
      < 0 &\mbox{for}& \beta \in \big(\mathrm{ln}(1), \mathrm{ln}(2)\big),
      \\
      = 0 &\mbox{for}& \beta = \mathrm{ln}(2),
      \\
      > 0 &\mbox{for}& \beta > \mathrm{ln}(2).
    \end{array}
    \right.
      $$
  It follows for the Cayley distance kernel on $\mathrm{Sym}(3)$ that:
  \begin{equation}
    \label{DefinitenessOfCayleyDistanceKernelOnSym3}
    \big[
      e^{- \beta \cdot d_C}
    \big]
    \;\;\mbox{is}\;\;
    \left\{
    \def\arraystretch{1.1}
    \begin{array}{ccl}
      \mbox{positive semi-definite} &\mbox{for}& \beta = \mathrm{ln}(1),
      \\
      \mbox{indefinite} &\mbox{for}& \beta \in \big(\mathrm{ln}(1), \mathrm{ln}(2)\big),
      \\
      \mbox{positive semi-definite} &\mbox{for}& \beta = \mathrm{ln}(2),
      \\
      \mbox{positive definite} &\mbox{for}& \beta > \mathrm{ln}(2).
    \end{array}
    \right.
  \end{equation}
\end{example}

In general, the spectrum of the Cayley distance kernel has an explicit expression
in terms of the representation theory of the symmetric group:

\begin{defn}[Irreducible characters of the symmetric groups]
  For $\lambda$ a partition of $N \in \mathbb{N}$, hence a weakly decreasing sequence of
  positive natural numbers that sum to $N$:
  \begin{equation}
    \label{Partition}
    \lambda
    \;=\;
    (\lambda_1 \geq \lambda_2 \geq  \cdots \geq \lambda_{\mathrm{rows}(\lambda)})
       \,,
       \;\;\;\;\;\;\;\;\;
    \underset{i}{\sum} \lambda_i \;=\; N
    \,,
    {\phantom{AAAAA}}
    \lambda_i \in \mathbb{N}_+
    \,,
  \end{equation}

  \vspace{-2mm}
\noindent  we write
  \begin{equation}
    \label{IrreducibleCharacter}
    \chi^{(\lambda)}(-)
    \;:=\;
    \mathrm{Tr}\big(S^{(\lambda)}(-)\big)
    \;:\;
    \mathrm{Sym}(N) \longrightarrow  \mathbb{C}
  \end{equation}
  for the irreducible character corresponding
  to the irreducible complex linear representation of
  $\mathrm{Sym}(N)$
  \begin{equation}
    \label{IrreducibleRepresentation}
    S^{(\lambda)} \;:\; \mathrm{Sym}(N) \longrightarrow  \mathrm{GL}(N,\mathbb{C})
  \end{equation}
  which is the {\it Specht module}
  labeled by $\lambda$ (e.g. \cite[\S 2.3]{Sagan01}).
\end{defn}

Observing that the exponentiated Cayley distance function from the origin
$$
  \sigma
    \;\longmapsto\;
  e^{- \beta \cdot d_C(e,\sigma)}
    \;=\;
  e^{- \beta \cdot N}
  e^{ \beta \cdot \#\mathrm{cycles}(\sigma) }
$$
is a class function
(manifestly so by Cayley's formula \eqref{CayleyFormula} used on the right),
the following Prop. \ref{CharacterFormulaForSpectrumOfCayleyDistanceKernel}
is the special case of a general character formula
for kernels on finite groups \cite[Thm. 1.1]{RKHS02}\cite[Cor. 5.4]{Kaski02}
(for just the eigenvalues this is due to \cite[Cor. 3]{DiaconisShahshahani81},
with a streamlined derivation given in \cite[Thm. 4.3]{FosterKriloff15},
generalizing a classical result for abelian groups \cite[Cor. 3.2]{Babai79}
following \cite{Lovasz75}):

\begin{prop}[Character formula for spectrum of Cayley distance kernel]
  \label{CharacterFormulaForSpectrumOfCayleyDistanceKernel}
  For all $N \in \mathbb{N}_+$ and $\beta \in \mathbb{R}_{\geq 0}$,
  the eigenvalues of the Cayley distance kernel $\big[ e^{- \beta \cdot d_C} \big]$
  are

\vspace{-2mm}
  \begin{enumerate}[{\bf (i)}]
  \setlength\itemsep{-2pt}
    \item
      indexed by the partitions $\lambda$ of $N$ \eqref{Partition};
    \item
      given by the formula
      \vspace{-3mm}
      \begin{equation}
        \label{CharacterFormulaForEigenvalues}
        \mathrm{EigVals}[e^{-\beta \cdot d_C}]_{\lambda}
        \;\;
          =
        \;\;
        \frac{
          e^{- \beta \cdot N}
        }{\chi^{(\lambda)}(e)}
        \underset{
          \sigma \in \mathrm{Sym}(N)
        }{\sum}
        e^{\beta \cdot \#\mathrm{cycles}(\sigma) }
        \cdot
        \chi^{(\lambda)}(\sigma)
        \,,
      \end{equation}

      \vspace{-3mm}
     \noindent
      where $\chi^{(\lambda)}$ is the corresponding irreducible character \eqref{IrreducibleCharacter};
    \item
     appearing with multiplicity $\big( \chi^{(\lambda)}(e)\big)^2$
     (the square of the dimension of the $\lambda$th irreducible representation);
    \item
      whose corresponding eigenvectors are the
      (complex conjugated) component functions
      $$
        \mathrm{EigVects}[e^{- \beta \cdot d_C}]_{\lambda, i, j}
        \;\;
          =
        \;\;
        \big(
          \bar S^{(\lambda)}_{i j}(\sigma)
        \big)_{\sigma \in \mathrm{Sym}(N)}
        \;\;
          \in
        \;\;
        \mathbb{C}
        \big(
          \mathrm{Sym}(N)
        \big)
      $$
      of the irreducible representations $S^{(\lambda)}$ \eqref{IrreducibleRepresentation}
      for all $1 \leq i,j \leq \chi^{(\lambda)}(e)$.
  \end{enumerate}
\end{prop}

\begin{example}[Eigenvalues of unit multiplicity]
  \label{EigenvaluesOfMultiplicity1}
  The Cayley distance kernel on $\mathrm{Sym}(N)$ has
  exactly two eigenvalues of multiplicity 1, namely
  the homogeneous distribution (corresponding to the trivial irrep of $\mathrm{Sym}(N)$)
  and the signature-distribution (corresponding to the sign irrep)
  whose eigenvalues are the sum
  (necessarily positive) and the signed sum
  (for $\mathrm{sgn}(\sigma)$ the signature), respectively, over any
  one row of the Cayley distance kernel matrix:
  \begin{equation}
  \label{EigenvectorsOfMultiplicityOne}
  \hspace{-1cm}
  \def\arraystretch{1.5}
  \begin{tabular}{|c||c|c|}
    \hline
    $\lambda$
      &
    $\mathrm{EigVects}[e^{- \beta \cdot d_C}]_{\lambda}$
      &
    $\mathrm{EigVals}[e^{- \beta \cdot d_C}]_{\lambda}$
    \\
    \hline
    \hline
    $(N)$
    &
    $
    \big(
      1
    \big)_{\sigma \in \mathrm{Sym}(N)}
    $
    &
    $
      \underset{\sigma \in \mathrm{Sym}(N)}{\sum}
      e^{- \beta d_c(e,\sigma)}
      \;>\;
      0
    $
    \\
    \hline
    $(1 \geq \cdots \geq 1)$
    &
    $
    \big(
      \mathrm{sgn}(\sigma)
    \big)_{\sigma \in \mathrm{Sym}(N)}
    $
    &
    $
      \underset{\sigma \in \mathrm{Sym}(N)}{\sum}
      \mathrm{sgn}(\sigma)
      \cdot
      e^{- \beta d_c(e,\sigma)}
    $
    \\
    \hline
  \end{tabular}
  \end{equation}
  This follows from Prop. \ref{CharacterFormulaForSpectrumOfCayleyDistanceKernel}
  since
  symmetric groups have exactly these two 1-dimensional irreps
  (e.g. \cite[Ex. 2.3.6, 2.3.7]{Sagan01}\cite[\S 7.B(2)]{Diaconis88});
  but in these simple cases the eigenvalues are also readily seen
  using the left-invariance property
  of the Cayley distance kernel (Lemma \ref{LeftInvarianceOfCayleyDistance}).
\end{example}

%%%%%%%%%%%%%%%%%%%%%%%%%%%%%%%%%%%%%%%%%%%%%%%%%%%%%%%%%%%%%%%%%%%%%%%%%%%%%%%%%%%%%%%%%%%%%
\subsection{Indefinite phases}
%%%%%%%%%%%%%%%%%%%%%%%%%%%%%%%%%%%%%%%%%%%%%%%%%%%%%%%%%%%%%%%%%%%%%%%%%%%%%%%%%%%%%%%%%%%%

\begin{lemma}[Cauchy interlace]
  \label{CauchyInterlace}
  Let $N_1 < N_2 \,\in \mathbb{N}$ and $\beta \in \mathbb{R}_{\geq 0}$.

  {\bf (i)} If $\big[e^{- \beta \cdot d_C}\big]$ is indefinite on $\mathrm{Sym}(N_1)$
  then it is indefinite on $\mathrm{Sym}(N_2)$.

  {\bf (ii)} If $\big[e^{- \beta \cdot d_C}\big]$ is positive definite on $\mathrm{Sym}(N_2)$
   then it is positive definite on $\mathrm{Sym}(N_1)$.
\end{lemma}
\begin{proof}
  Since the Cayley distance kernel on $\mathrm{Sym}(N_1)$
  is a principal submatrix of that on $\mathrm{Sym}(N_2)$,
  by Lemma \ref{CayleyDistanceIsPreservedByInclusionOfSymmetricGroups},
  this follows from the general fact that for
  $A_1$ a principal submatrix of a symmetric hermitian matrix $A_2$,
  their lowest eigenvalues satisfy
  $$
    \mathrm{min}\big( \mathrm{EigVals}(A_2) \big)
    \;\leq\;
    \mathrm{min}\big( \mathrm{EigVals}(A_1) \big)
  $$
  (a direct consequence of
  \scalebox{.9}{$\mathrm{min}\big( \mathrm{EigVals}(A) \big)
    \!=\!
  \underset{\left\vert v \right\vert = 1 }{\mathrm{min}} \langle \bar v, A v \rangle$},
  and a simple case of Cauchy's interlace theorem, e.g. \cite{Hwang04}).
\end{proof}

\begin{prop}[Cayley distance kernel indefinite for $0 < \beta < \mathrm{ln}(2)$]
  \label{CayleyDistanceKernelNotPositiveForBetaSmallerLn2}
  For all $N$ (Nota. \ref{TheParameters}), the Cayley distance
  kernel ceases to be positive semi-definite as soon as $\beta < \mathrm{ln}(2)$:
  $$
    0 < \beta < \mathrm{ln}(2)
    \;\;\;\;\;\;\;
    \Rightarrow
    \;\;\;\;\;\;\;
    \underset{N \geq 2}{\forall}
    \mbox{
      \rm
      $e^{- \beta \cdot d_C}$
      is indefinite
      on $\mathrm{Sym}(N)$
    }
    .
  $$
\end{prop}
\begin{proof}
  Use Example \ref{CayleyDistanceKernelOnSym3}
  in Lemma \ref{CauchyInterlace}.
\end{proof}

To proceed further, we need the following two results from enumerative combinatorics:
\begin{lemma}[First polynomial relation, {e.g. \cite[Prop. 1.3.7]{Stanley86}}]
\label{CuriousLemma}
The following holds as an equation of polynomials in $e^\beta$:
\begin{equation}
  \label{FirstPolynomialRelation}
  \underset{
    \sigma \in \mathrm{Sym}(N)
  }{
    \sum
  }
  e^{\beta \cdot \#\mathrm{cycles}(\sigma)}
  \;\;=\;\;
  \underoverset
    {k = 0}
    {N - 1}
    {\prod}
  \big(
    e^\beta
    +
    k
  \big)
\,.
\end{equation}
\end{lemma}
\begin{proof}
  Identify permutations with the marked lists  underlying their
  unique representatives in cycle notation, for which heads of cycles
  are the smallest elements in their cycle and cycles are ordered by their heads.
  Then observe that, in this guise, permutations are manifestly enumerated,
  starting from the empty such list, by iteratively, over $k = 1,2,3, \cdots, N$,
  including the element $k+1$ into the list,
  either adjoined to the right of the list if it is the head of a cycle
  (in which case it contributes a factor $e^\beta$ to $e^{\beta \cdot \# \mathrm{cycles}}$),
  or else inserted after one of the $k$ elements already in the list
  (in which case it contributes a factor of 1).
\end{proof}
\begin{lemma}[Second polynomial relation]
\label{CuriousLemmaSigned}
The following holds as an equation of polynomials in $e^\beta$:
\begin{equation}
  \label{SecondPolynomialEquation}
  \underset{
    \sigma \in \mathrm{Sym}(N)
  }{
    \sum
  }
  \mathrm{sgn}(\sigma)
    \cdot
  e^{\beta \cdot \#\mathrm{cycles}(\sigma)}
  \;\;=\;\;
  \underoverset
    {k = 0}
    {N - 1}
    {\prod}
  \big(
    e^\beta
    -
    k
  \big)
\,,
\end{equation}
where $\mathrm{sgn}(\sigma)$ denotes the signature of a permutation.
\end{lemma}
\begin{proof}
  We compute as follows:
$$
  \begin{aligned}
    \underoverset
      {k = 0}
      {N - 1}
      {\prod}
    \big(
      e^\beta - k
    \big)
    &
    \;=\;
    (-1)^N
    \underoverset
      {k = 0}
      {N - 1}
      {\prod}
      \big(
        -e^\beta + k
      \big)
    \\
    &
    \;=\;
    (-1)^N
    \underset{
      \sigma \in \mathrm{Sym}(N)
    }{\sum}
    (-e^\beta)^{ \#\mathrm{cycles}(\sigma) }
    \\
    &
    \;=\;
    \underset{
      \sigma \in \mathrm{Sym}(N)
    }{\sum}
    (-1)^{ N \,+\, \#\mathrm{cycles}(\sigma) }
    \cdot
    e^{\beta \cdot \#\mathrm{cycles}(\sigma) }
    \\
    &
    \;=\;
    \underset{
      \sigma \in \mathrm{Sym}(N)
    }{\sum}
    (-1)^{ \#\mathrm{cycles}_{ \scalebox{.4}{ev.length}  }(\sigma) }
    \cdot
    e^{\beta \cdot \#\mathrm{cycles}(\sigma) }
    \\
    &
    \;=\;
    \underset{
      \sigma \in \mathrm{Sym}(N)
    }{\sum}
    \mathrm{sgn}(\sigma)
    \cdot
    e^{\beta \cdot \#\mathrm{cycles}(\sigma) }
    \,.
  \end{aligned}
$$
Here the second step is Lemma \ref{CuriousLemma}.
The fourth step observes that if $N$ is even/odd, then there
must be an even/odd number of permutations of odd length,
so that the sign that remains is given by the number of permutations of
even length. But this is the signature, since
the number of transpositions making a cycle is one less than its length
\eqref{CyclicPermutationAsCompositeOfTranspositions}.
\end{proof}

Using this we may improve the characterization of indefiniteness in
Proposition \ref{CayleyDistanceKernelNotPositiveForBetaSmallerLn2}:
\begin{prop}
  \label{CayleyDistanceKernelIndefiniteAtNonIntegerExpInverseTemperatureBelowNMinusOne}
  The Cayley distance kernel
  on $\mathrm{Sym}(N)$ is indefinite for $e^\beta$ below $N-1$ and non-integer:
  \begin{equation}
    \label{IntervalsOfIndefiniteness}
    e^{\beta}
    \;\in\;
    (0,1) \cup (1,2) \cup \cdots (N-2,N-1)
    {\phantom{AAAA}}
      \Rightarrow
    {\phantom{AAAA}}
    \mbox{
      $e^{- \beta \cdot d_C}$
      is indefinite.
    }
  \end{equation}
\end{prop}
\begin{proof}
  Observe that the Cayley distance kernel has the following eigenvalue
  \begin{equation}
    \label{EigenvalueOfSignatureDistribution}
    \begin{aligned}
      \mathrm{EigVals}[e^{-\beta \cdot d_C}]_{(1 \geq  \cdots \geq 1)}
      &
      \;=\;
      e^{- \beta \cdot N}
      \underset{
        \sigma \in \mathrm{Sym}(N)
      }{\sum}
      \mathrm{sgn}(\sigma)
      \cdot
      e^{ \beta \cdot \#\mathrm{cycles}(\sigma) }
      \\
      &
      \;=\;
      e^{- \beta \cdot N}
      \;\;
      \underoverset
        {k = 0}
        {N - 1}
        {\prod}
      \big(
        e^\beta - k
      \big)
      \,.
    \end{aligned}
  \end{equation}
  Here the first line
  is Example \ref{EigenvaluesOfMultiplicity1}
  expressed using Cayley's formula \eqref{CayleyFormula},
  and the second step is Lemma \ref{CuriousLemmaSigned}.
  Hence we find that the Cayley distance kernel always has an eigenvalue of
  the following sign:
  \begin{equation}
    \label{SignOfEigenvalueOfSignatureDistribution}
    \underoverset
      {k = 0}
      {N - 1}
      {\prod}
    \big(
      e^\beta - k
    \big)
    \;\;
    \mbox{is}
    \;
    \left\{
    \begin{array}{ccl}
      > 0 &\mbox{for}& e^\beta > N - 1,
      \\
      = 0 &\mbox{for}& e^\beta \in \{0,1, \cdots, N-1\},
      \\
      < 0 &\mbox{for}& e^{\beta} \in \cdots (N-4, N-3) \cup (N-2,N-1).
    \end{array}
    \right.
  \end{equation}
  Here the first two lines are immediate from the form of the polynomial;
  and with this the third line
  follows by observing that all roots of the polynomial have unit multiplicity, so that
  its sign must change whenever $e^{\beta}$ crosses one of its zeros.

  This shows that the Cayley distance kernel
  on $\mathrm{Sym}(N)$ has a negative eigenvalue
  at least on every
  {\it second} of the
  open intervals claimed. But the same argument applies to the
  kernel on $\mathrm{Sym}(N-1)$, to show that this has a negative eigenvalue on
  every {\it other}, remaining, open interval.
  Since
  the latter kernel is
  (by Lemma \ref{CayleyDistanceIsPreservedByInclusionOfSymmetricGroups})
  a principal submatrix of the former,
  Lemma \ref{CauchyInterlace}
  implies that the kernel
  on $\mathrm{Sym}(N)$
  has a negative eigenvalue on all the open intervals \eqref{IntervalsOfIndefiniteness},
  as claimed.
\end{proof}

%%%%%%%%%%%%%%%%%%%%%%%%%%%%%%%%%%%%%%%%%%%%%%%%%%%%%%%%%%%%%%%%%%%%%%%%%
\subsection{Semi-definite phases}
%%%%%%%%%%%%%%%%%%%%%%%%%%%%%%%%%%%%%%%%%%%%%%%%%%%%%%%%%%%%%%%%%%%%%%%%%

%$$
%  \begin{tikzcd}[row sep=0pt]
%    \mathrm{Sym}(n)
%    \ar[
%      rr,
%      "\ell\mathrm{cycles}"
%    ]
%    &&
%    \mathrm{Part}(n)
%    \\
%    \sigma
%      &\mapsto&
%    \big(
%      \ell_1(\sigma)
%        \geq
%      \ell_2(\sigma)
%        \geq
%      \cdots
%        \geq
%      \ell_{\#\mathrm{cycles}(\sigma)}
%    \big)
%  \end{tikzcd}
%$$
%for the function that sends a permutation to the sequence of the lengths of its
%cycles, ordered by decreasing size and thus regarded as a partition \eqref{Partition} of $N$.

For our proof of the exceptional positive semi-definite phases of
the Cayley distance kernel in Prop. \ref{CayleyDistanceKernelPositiveSemiDefiniteAtLowIntegers} below, we need Frobenius' character formula for Schur polynomials,
recalled as Prop. \ref{FrobeniusFormula} below, and we need to know that
Schur polynomials count semistandard Young tableaux:

\begin{defn}[Schur polynomials (e.g. {\cite[Def. 4.4.1]{Sagan01}})]
  \label{SchurPolynomial}
  For $\lambda = (\lambda_1 \geq \cdots \geq \lambda_{\mathrm{rows}(\lambda)})$ a partition \eqref{Partition}
  \vspace{-2mm}
  \begin{enumerate}[{\bf (i)}]
  \setlength\itemsep{-2pt}
  \item
  a {\it semistandard Young tableau} $T$
  of {\it shape} $\left\vert T \right\vert = \lambda$
  is an array
  $
    \big(
    T_{i, j}
    \,
    \in
    \,
    \mathbb{N}_+
    \big)_{
      {1 \leq i \leq \mathrm{rows}(\lambda)}
      \atop
      { 1 \leq j \leq \lambda_i}
    }
  $
  of positive natural numbers
  such that $j_1 < j_2 \Rightarrow T_{i j_1} \leq T_{i j_2}$
  and $i_1 < i_2 \Rightarrow T_{i_1 j} < T_{i_2 j}$.
  We write
  \begin{equation}
    \label{SetsOfSemistandardYoungTableaux}
    \mathrm{ssYT}_\lambda
    \;\supset\;
    \mathrm{ssYT}_\lambda\!(n)
    \;\subset\;
    \mathrm{ssYT}_N(n)
  \end{equation}
  for, respectively,
  the sets of all ssYT of shape $\lambda$,
  with labels $T_{i,j} \leq n$,
  among all those with $N = \sum_i \lambda_i$ boxes.

 \item
   The {\it monomial} corresponding to an ssYT in the polynomial ring
   on a countable number of generators
  is
  \vspace{-1mm}
  $$
    x^T \;\coloneqq\; x^{\#1s(T)} x^{\# 2s(T)} \cdots
  $$

    \vspace{-4mm}
\noindent  with $\#1s(T)$ denoting the number of entries of $T$ labeled with the value 1, etc.

\item The {\it Schur polynomial} $s_\lambda$ in $n$ variables,
  indexed by the partition $\lambda$, is the
  sum of these monomials over all semistandard Young tableaux $T$ whose shape
  $\left\vert T \right \vert$ is $\lambda$ and whose labels are bounded as
  $T_{i,j} \leq n$:
  \vspace{-3mm}
  \begin{equation}
    \label{SchurPolynomialViaSumOverSemistandardYoungTableau}
    s_\lambda
    \big(
      x_1, x_2, \cdots, x_n
    \big)
    \;\;
    =
    \;\;
    \underset{
        T \in \mathrm{ssYT}_\lambda\!(n)`
    }{\sum}
    x^T
    \,.
  \end{equation}
  In particular, the value
  \begin{equation}
    \label{SchurPolynomialAtxiEqual1}
    s_\lambda
    \big(
      x_1 \!=\! 1,
      \cdots,
      x_n \!=\! 1
    \big)
    \;\;
    =
    \;\;
    \underset{
      T \in \mathrm{ssYT}_\lambda\!(n)
    }{\sum}
    \!\!\!\!
    1
    \;\;=\;\;
    \# \mathrm{ssYT}_\lambda\!(n)
  \end{equation}
  is the number of semistandard Young tableaux of shape $\lambda$ with labels
  $\leq n$.
  \end{enumerate}
\end{defn}
\begin{example}
  \label{PositivityOfNumberOfSemiStandardYoungTableaux}
  If $\lambda = (\lambda_1 \geq \cdots \geq_{\mathrm{rows}(\lambda)})$
  is a partition of $\underset{i}{\sum} \lambda_i = N$,
  then Def. \ref{SchurPolynomial} yields:
  $$
    s_\lambda
    \big(
      x_1 \!=\! 1,
      \cdots,
      x_{\mathrm{rows}(\lambda)} \!=\! 1
    \big)
    \;
      =
    \;
    \# \mathrm{ssYT}_\lambda\!(n)
    \;\;\mbox{is}\;\;
    \left\{
    \begin{array}{lcl}
      = 0 &\mbox{if}& n < \mathrm{rows}(\lambda)
      \\
      > 0 &\mbox{if}& n \geq \mathrm{rows}(\lambda)
      \,,
    \end{array}
    \right.
  $$
  because:

  \noindent
  1) by vertical strict monotonicity, an
  ssY tableau of shape $\lambda$ needs
  $n \geq \mathrm{rows}(\lambda)$ labels to fill its first column,

  \noindent
  2) while the weak horizontal monotonicity allows any labelling of the
   first column to be completed to all columns.

\end{example}

\begin{prop}[Character formula for Schur polynomials {\cite[Thm. 4.6.4]{Sagan01}}]
  \label{FrobeniusFormula}
  For $N \in \mathbb{N}_+$ and $\lambda$ a partition of $N$ \eqref{Partition},
  the Schur polynomial $s_\lambda$ \eqref{SchurPolynomialViaSumOverSemistandardYoungTableau}
  may be expressed as follows:
      \vspace{-2mm}
  \begin{equation}
    \label{SchurPolynomialViaFrobeniusFormula}
    \begin{aligned}
    &
    s_\lambda
    \big(
      x_1, x_2, \cdots, x_n
    \big)
    \\
    &
    \;=\;
    \tfrac{1}{N!}
    \underset{
      \sigma \in \mathrm{Sym}(N)
    }{\sum}
    \chi^{(\lambda)}(\sigma)
    \cdot
    \big(
      x_1^{\ell_1(\sigma)}
      +
      \cdots
      +
      x_n^{\ell_1(\sigma)}
    \big)
    \big(
      x_1^{\ell_2(\sigma)}
      +
      \cdots
      +
      x_n^{\ell_2(\sigma)}
    \big)
    \cdots
    \big(
      x_1^{\ell_{\#\mathrm{cycles}(\sigma)}(\sigma)}
      +
      \cdots
      +
      x_n^{\ell_{\#\mathrm{cycles}(\sigma)}(\sigma)}
    \big)
    \,,
    \end{aligned}
  \end{equation}

  \vspace{-2mm}
  \noindent
  where $\chi^{(\lambda)}$ denotes the $\lambda$th irreducible character \eqref{IrreducibleCharacter} and
  $\ell_k(\sigma)$ denotes the length of the $k$th longest cycle of $\sigma$.
\end{prop}

\begin{lemma}[Eigenvalues of Cayley distance kernel at $e^{\beta} \in \mathbb{N}_+$ count
 semi-stable Young tableaux]
\label{EigenvaluesOfCayleyDistanceKernelCountSemiStandardYoungTableau}
  $\,$

  \noindent
  For $e^\beta = n \in \mathbb{N}_+$, the eigenvalues
  \eqref{CharacterFormulaForEigenvalues} of the
  Cayley distance kernel
  count semi-standard Young tableaux (Def. \ref{SchurPolynomial}), in that
  for any partition \eqref{Partition} $\lambda$ of $N$ we have:
  \begin{equation}
    \mathrm{EigVals}[e^{- \mathrm{ln}(n) \cdot d_C}]_\lambda
    \;\;=\;\;
    \frac{1}{n^N}
    \frac{
      N !
    }{
      \chi^{(\lambda)}(e)
    }
    \cdot
    \# \mathrm{ssYT}_\lambda\!(n)
    \,.
  \end{equation}
\end{lemma}
\begin{proof}
  We compute as follows:
  $$
    \begin{aligned}
      \mathrm{EigVals}[e^{-\mathrm{ln}(n) \cdot d_C}]_\lambda
      &
      \;=\;
      \frac{
        e^{- \mathrm{ln}(n)\cdot N}
      }{
        \chi^{(\lambda)}(e)
      }
      \underset{\sigma \in \mathrm{Sym}(N)}{\sum}
      \chi^{(\lambda)}(\sigma)
        \cdot
      n^{\# \mathrm{cycles}(\sigma)}
      \\
      &
      \;=\;
      \frac{
        e^{- \mathrm{ln}(n)\cdot N}
      }{
        \chi^{(\lambda)}(e)
      }
      \underset{\sigma \in \mathrm{Sym}(N)}{\sum}
      \chi^{(\lambda)}(\sigma)
        \cdot
      \big(
        n \cdot 1^{\ell_1(\sigma)}
      \big)
      \big(
        n \cdot 1^{\ell_2(\sigma)}
      \big)
      \cdots
      \big(
        n \cdot 1^{\ell_{\#\mathrm{cycles}(\sigma)}(\sigma)}
      \big)
      \\
      &
      \;=\;
      N!
      \frac{
        e^{- \mathrm{ln}(n)\cdot N}
      }{
        \chi^{(\lambda)}(e)
      }
      \cdot
      s_\lambda
      \big(
        x_1 \!=\! 1,
        \cdots,
        x_n \!=\! 1
      \big)
      \\
      &
      \;=\;
      N!
      \frac{
        e^{- \mathrm{ln}(n)\cdot N}
      }{
        \chi^{(\lambda)}(e)
      }
      \,
      \cdot
      \,
      \# \mathrm{ssYT}_\lambda\!(n)
      \,.
    \end{aligned}
  $$

  \vspace{-1mm}
  \noindent
  Here the first line is the
  character formula for kernel spectra \eqref{CharacterFormulaForEigenvalues},
  the third step is the character formula for Schur polynomials \eqref{SchurPolynomialViaFrobeniusFormula}
  and the last step is their counting property \eqref{SchurPolynomialAtxiEqual1}.
\end{proof}
\begin{remark}[Alternative proof]
  Alternatively, with the closely related
  formula \cite[Prop. 2.4]{GnedinGorinKerov11}
  \begin{equation}
    \label{GnedinGorinKerovFormula}
    n^{\#\mathrm{cycles}(-)}
    \;:=\;
    \underset{
      T \in \mathrm{ssYT}_N(n)
    }{\sum}
    \chi^{ (\left\vert T \right\vert)}(-)
    \,,
  \end{equation}
  Lemma \ref{EigenvaluesOfCayleyDistanceKernelCountSemiStandardYoungTableau}
  follows by the following computation:
  $$
  \begin{aligned}
    \mathrm{EigVals}[e^{- \mathrm{ln}(n) \cdot d_C}]_\lambda
    &
    \;=\;
    \tfrac{e^{-\mathrm{ln}(n) \cdot N}}{\chi^{(\lambda)}(e)}
    \underset{\sigma \in \mathrm{Sym}(N)}{\sum}
    n^{ \# \mathrm{Cycles}(\sigma)   }
    \cdot
    \bar \chi^{(\lambda)}(\sigma)
    \\
    &
    \;=\;
    \tfrac{e^{-\mathrm{ln}(n) \cdot N}}{\chi^{(\lambda)}(e)}
    \underset{
      {T \in}
      \atop
      {\mathrm{ssYT}_N(n)}
    }{\sum}
    \;
    \underset{
      {\sigma \in}
      \atop
      {\mathrm{Sym}(n)}
    }{\sum}
    \chi^{ \left\vert T \right\vert }(\sigma)
    \cdot
    \bar \chi^{(\lambda)}(\sigma)
    \\
    &
    \;=\;
    \tfrac{n!/n^N}{\chi^{(\lambda)}(e)}
    \underset{
      {T \in}
      \atop
      {  \mathrm{ssYT}_N(n)  }
    }{\sum}
    \delta^{ \left\vert T\right\vert, \lambda }
    \,.
  \end{aligned}
$$
Here the first line is again the character formula \eqref{CharacterFormulaForEigenvalues}
from Prop. \ref{CharacterFormulaForSpectrumOfCayleyDistanceKernel},
shown under complex conjugation (which does not change the real eigenvalue).
The second step inserts \eqref{GnedinGorinKerovFormula}
and the last step applies Schur orthogonality (e.g. \cite[Thm. 2.12]{FultonHarris91}).
\end{remark}

\begin{prop}[Positivity of Cayley distance kernel at log-integral inverse temperature]
  \label{CayleyDistanceKernelPositiveSemiDefiniteAtLowIntegers}
  $\,$

  \noindent
  For $e^\beta := n \in \{1, 2, \cdots, N-1\}$
  the Cayley distance kernel $e^{- \beta \cdot d_C}$ on $\mathrm{Sym}(N)$
  is positive semi-definite,
  while for $e^\beta \in \{N, N+1, \cdots \}$ it is positive definite.
  Hence \eqref{FundamentalWeightSystemIsQuantumStateIffCayleyDistanceKernelIsPositiveSemiDefinite}
  all fundamental weight systems are quantum states.
\end{prop}
\begin{proof}
  By Lemma \ref{EigenvaluesOfCayleyDistanceKernelCountSemiStandardYoungTableau},
  all eigenvalues at these temperatures are non-negative,
  and with Example \ref{PositivityOfNumberOfSemiStandardYoungTableaux}
  all are positive for $n \geq N$, while for $n < N$
  at least the eigenvalue labeled by the
  sign representation $\lambda = (1 \geq \cdots \geq 1)$ \eqref{EigenvaluesOfMultiplicity1}
  takes the value 0
  (as seen already in \eqref{SignOfEigenvalueOfSignatureDistribution}).
\end{proof}

%%%%%%%%%%%%%%%%%%%%%%%%%%%%%%%%%%%%%%%%%%%%%%%%%%%%%%%%%%%%%%%%%%%%%%%%%%
\subsection{Definite phase}
%%%%%%%%%%%%%%%%%%%%%%%%%%%%%%%%%%%%%%%%%%%%%%%%%%%%%%%%%%%%%%%%%%%%%%%%%%

We establish a sharp lower bound for the inverse temperature $\beta$ above which the
Cayley distance kernel is always positive definite
(Prop. \ref{CayleyDistanceKernelPositiveDefiniteForSufficientlyLargeBeta} below;
for log-integral inverse tempteratures this is already the statement of Prop. \ref{CayleyDistanceKernelPositiveSemiDefiniteAtLowIntegers}).
The argument
via Stanley's combinatorial {\it hook-content formula}
in the following Lemma \ref{CayleyDistanceKernelEigenvaluesViaHookContent}
was kindly pointed out to us by A. Abdesselam;
it improves on an earlier proof of ours of a loose lower bound via
the Gershgorin circle theorem.

\begin{prop}[Hook-content formula {\cite[Thm. 15.3]{Stanley71}\cite[Thm. 7.21.2]{Stanley99}}]
  \label{HookContentFormula}
  The number of semi-standard Young tableau \eqref{SetsOfSemistandardYoungTableaux}
  of shape $\lambda$ and
  with labels $\leq n$ is given by
  $$
    \# \mathrm{ssYT}_\lambda(n)
    \;\;=\;\;
    \underset{
      { 1 \leq i \leq \mathrm{rows}(\lambda) }
      \atop
      { 1 \leq j \leq \lambda_j }
    }{\prod}
    \frac{
      n + j - i
    }{
      \ell \mathrm{hook}_\lambda\!(i,j)
    }
    \,,
  $$
\end{prop}
\noindent
where the product is over all boxes $(i,j)$ of the underlying Young diagram,
and $\ell \mathrm{hook}_\lambda\!(i,j) \in \mathbb{N}_+$ denotes the ``hook length''
at position $(i,j)$, hence the sum of the number of boxes to the right
and below the box, plus one for the box itself.

\begin{lemma}[Eigenvalues of Cayley distance kenrel in terms of hook-content]
  \label{CayleyDistanceKernelEigenvaluesViaHookContent}
  For all $\beta \in \mathbb{R}_+$, the $\lambda$th eigenvalue
  \eqref{CharacterFormulaForEigenvalues}
  of the Cayley distance kernel on $\mathrm{Sym}(N)$ is given by
  \begin{equation}
    \label{CayleyDistanceKernelEigenvaluesExpressedViaHookContent}
    \mathrm{EigVals}[e^{- \beta \cdot d_C}]_\lambda
    \;=\;
    \frac
      { N! }
      {\chi^{(\lambda)}(e) }
    e^{ - \beta \cdot N }
    \underset{
      { 1 \leq i \leq \mathrm{rows}(\lambda) }
      \atop
      { 1 \leq j \leq \lambda_j }
    }{\prod}
    \frac
      { e^\beta + j - i }
      { \ell\mathrm{hook}_\lambda\!(i,j) }
    \,.
  \end{equation}
\end{lemma}
\begin{proof}
  Observe that for $e^\beta = n \in \mathbb{N}_+$
  we have
  \begin{equation}
    \label{EqualityOfEigenvaluesWithHookContentForLogIntegralInverseTemperature}
    \frac
      {\chi^{(\lambda)}(e)}
      {N!}
    \cdot
    n^N
    \cdot
    \mathrm{EigVals}[e^{-\mathrm{ln}(n)\cdot d_C}]_\lambda
    \;=\;
    \#\mathrm{ssYT}_\lambda(n)
    \;=\;
    \underset{
      { 1 \leq i \leq \mathrm{rows}(\lambda) }
      \atop
      { 1 \leq j \leq \lambda_j }
    }{\prod}
    \frac{
      n + j - i
    }{
      \ell \mathrm{hook}_\lambda\!(i,j)
    }
    \,,
  \end{equation}
  where the first equation is from Lemma \ref{EigenvaluesOfCayleyDistanceKernelCountSemiStandardYoungTableau}
  and the second equation from Prop. \ref{HookContentFormula}.
  But this equation \eqref{EqualityOfEigenvaluesWithHookContentForLogIntegralInverseTemperature}
  is the specialization to integral values
  $e^\beta = n$ of the following more general equation,
  which is equivalent to the equation \eqref{CayleyDistanceKernelEigenvaluesExpressedViaHookContent}
  that we need to prove:
  \begin{equation}
    \label{GeneralEqualityOfEigenvaluesWithHookContent}
    \frac
      {\chi^{(\lambda)}}
      {N!}
    \cdot
    (e^\beta)^N
    \cdot
    \mathrm{EigVals}[e^{-\mathrm{ln}(n)\cdot d_C}]_\lambda
    \;=\;
    \underset{
      { 1 \leq i \leq \mathrm{rows}(\lambda) }
      \atop
      { 1 \leq j \leq \lambda_j }
    }{\prod}
    \frac{
      e^{\beta} + j - i
    }{
      \ell \mathrm{hook}_\lambda\!(i,j)
    }
    \,.
  \end{equation}
  Now observing that both sides of \eqref{GeneralEqualityOfEigenvaluesWithHookContent}
  are polynomials in $e^\beta$ -- the right side manifestly so and the
  left hand side by \eqref{CharacterFormulaForEigenvalues} --
  the general equality follows already by the fact that it holds at infinitely many
  distinct values,
  by \eqref{EqualityOfEigenvaluesWithHookContentForLogIntegralInverseTemperature}.
\end{proof}
\begin{prop}\label{CayleyDistanceKernelPositiveDefiniteForSufficientlyLargeBeta}
  The Cayley distance kernel $e^{- \beta \cdot d_C}$
  on $\mathrm{Sym}(N)$ is positive definite for all $e^\beta > N - 1$.
\end{prop}
\begin{proof}
  From Lemma \ref{CayleyDistanceKernelEigenvaluesViaHookContent}
  it is manifest that all eigenvalues are positive as soon as
  $
    e^\beta
    >
    \!\!\!\!
    \underset{
      {\lambda \in \mathrm{Part}(N)}
      \atop
      { 1 \leq i \leq \mathrm{rows}(\lambda) }
    }{\mathrm{max}}
    \!\!\!
    (i - 1)
    \;=\;
    N - 1
    \,.
  $
\end{proof}

\medskip

%%%%%%%%%%%%%%%%%%%%%%%%%%%%%%%%%%%%%
\subsection{Explicit eigenvalues}
%%%%%%%%%%%%%%%%%%%%%%%%%%%%%%%%%%%%%

As a byproduct, we have now obtained 
fully explicit polynomial expressions for all the eigenvectors of the
Cayley distance kernel,
using the following classical fact from
combinatorial representation theory:
\begin{prop}[Hook formulas for symmetric and linear groups -- e.g. {\cite[Thm. 20.1]{James78} and \cite[(C.27)]{Sternberg94}}]
  \label{HookFormulasForDimensionsOfIrrepsOfSymmetricAndLinearGroups}
  $\,$

  \noindent
  For $\lambda$ a partition \eqref{Partition} of $N \in \mathbb{N}_+$,
  with corresponding complex irrep $S^{(\lambda)}$  of $\mathrm{Sym}(N)$
  \eqref{IrreducibleRepresentation}
  and corresponding complex irrep $V^{(\lambda)}$ of
  $\mathrm{SL}(N, \mathbb{C})$,
  hence polynomial complex irrep of $\mathrm{GL}(n,\mathbb{C})$
  (e.g. \cite[p. 114]{Fulton97} with \cite[\S 5.8]{Sternberg94}),
  their dimensions are given by the
  {\it hook length formula} and the
  {\it hook-content formula}
  (Prop. \ref{HookContentFormula}), as:

  \begin{equation}
  \mbox{
  \def\arraystretch{2}
  \begin{tabular}{|c|c|}
    \hline
    \rm
    hook length formula for $\mathrm{Sym}(N)$
    &
    \rm
    hook-content formula for $\mathrm{SL}(N,\mathbb{C})$
    \\
    \hline
    \hline
    $
    \mathrm{dim}_{\mathbb{C}}
    \big(
      S^{(\lambda)}
    \big)
    \;=\;
    N!
    \underset{
      { 1 \leq i \leq \mathrm{rows}(\lambda) }
      \atop
      { 1 \leq j \leq \lambda_i }
    }{\prod}
    \frac{1}{ \ell \mathrm{hook}_\lambda(i,j) }
  $
  &
  $
    \mathrm{dim}_{\mathbb{C}}
    \big(
      V^{(\lambda)}
    \big)
    \;=\;
    \underset{
      { 1 \leq i \leq \mathrm{rows}(\lambda) }
      \atop
      { 1 \leq j \leq \lambda_i }
    }{\prod}
    \frac{
      N + j - i
    }{ \ell \mathrm{hook}_\lambda(i,j) }
  $
  \\
  \hline
  \end{tabular}
  }
  \end{equation}
\end{prop}
\noindent
\begin{prop}[Explicit eigenvalues of the Cayley distance kernel]
  \label{ExplicitEigenvaluesOfCayleyDistanceKernel}
  The eigenvalues \eqref{CharacterFormulaForEigenvalues} of the Cayley distance kernel
  $e^{- \beta \cdot d_C}$ on $\mathrm{Sym}(N)$
  have the following equivalent expressions:
  \begin{equation}
   \label{ExplicitFormulaForEigenvectorsOfCayleyDistanceKernel}
   \begin{aligned}
     \mathrm{EigVals}[e^{- \beta \cdot d_C}]_\lambda
     \;\;\;=\;\;\;
     e^{- \beta \cdot N}
    \underset{
      { 1 \leq i \leq \mathrm{rows}(\lambda) }
      \atop
      { 1 \leq j \leq \lambda_i }
    }{\prod}
    \big(
      e^\beta + j - i
    \big)
     \;\;\;\;\;\;\;\;
       \underset{
         \mathclap{
         e^\beta = N
         }
       }{=}
     \;\;\;\;\;\;\;\;
     \frac{N!}{ e^{\beta \cdot N} }
     \cdot
     \frac{
       \mathrm{dim}_{\mathbb{C}}\big( V^{(\lambda)} \big)
     }{
       \mathrm{dim}_{\mathbb{C}}\big( S^{(\lambda)} \big)
     }
    \,.
   \end{aligned}
  \end{equation}
\end{prop}
\begin{proof}
Use Prop. \ref{HookFormulasForDimensionsOfIrrepsOfSymmetricAndLinearGroups}
with
Lemma \ref{CayleyDistanceKernelEigenvaluesViaHookContent}.
\end{proof}
\begin{remark}[Alternative proof via Jucys-Murphy theory]
  The first statement in Prop. \ref{ExplicitEigenvaluesOfCayleyDistanceKernel}
  may be obtained alternatively as follows (we thank D. Speyer for pointing this out):
  In terms of the {\it Jucys-Murphy elements} in the group algebra
  $$
    J_k
    \;:=\;
    \underset{
      1 \leq i < n
    }{\sum}
    (i,k)
    \;\;
    \in
    \;
    \mathbb{C}[\mathrm{Sym}(n)]
  $$
  the Cayley distance kernel, regarded as a linear endomorphism on
  $\mathbb{C}[\mathrm{Sym}(n)]$, may be factored as
  \begin{equation}
    \label{CayleyDistanceKernelFactoredIntoJucysMurphyElements}
    e^{ \beta N }
    [e^{- \beta \cdot d_C}]
    \;\;
    =
    \big(
      e^{\beta} + J_1
    \big)
      \cdot
    \big(
      e^{\beta} + J_2
    \big)
      \cdots
    \big(
      e^{\beta} + J_n
    \big)
      \cdot
    \;\;\;\;
    \in
    \;
    \mathrm{End}
    \big(
      \mathbb{C}[\mathrm{Sym}(n)]
    \big)
  \end{equation}
  (as one sees inductively by representing permutations as
  products over sequences of transpositions making cycles as in \eqref{CyclicPermutationAsCompositeOfTranspositions}).
  But by \cite{Jucys71} (recalled in \cite[(12)]{Jucys74})
  and \cite[(3.18)]{Murphy81},
  the $J_k$ have a joint basis of eigenvectors
  $v_{T, \mu}$ labeled, in particular, by standard Young tableau $T \in \mathrm{sYT}_N$,
  with joint eigenvalues equal to
  $$
    \mathrm{EigVals}[J_k]_{T, \mu}
    \;=\;
    (j - i)
    \;\;\;\;\;
    \mbox{for $T_{i,j} = k$}
    \,.
  $$
  Plugging this into \eqref{CayleyDistanceKernelFactoredIntoJucysMurphyElements},
  using that in a standard Young tableaux every element in
  $\{1, \cdots, N\}$ appears exactly once as a label,
  yields the first form of \eqref{ExplicitFormulaForEigenvectorsOfCayleyDistanceKernel}.
\end{remark}
\begin{example}
  For $\lambda = (N)$ and $\lambda = (1 \geq \cdots \geq 1)$
  (Example \ref{EigenvaluesOfMultiplicity1}),
  equation \eqref{ExplicitFormulaForEigenvectorsOfCayleyDistanceKernel}
  reproduces the expressions \eqref{FirstPolynomialRelation}
  and \eqref{SecondPolynomialEquation}, respectively.
\end{example}

\medskip

%%%%%%%%%%%%%%%%%%%%%%%%%%%%%%%%%%%%%
\section{Weight systems as quantum states}
\label{PhysicsInterpretation}
%%%%%%%%%%%%%%%%%%%%%%%%%%%%%%%%%%%%

We close by briefly explaining the impact of
Thm. \ref{FundamentalglNWeightSystemsAreQuantumStates} on
current questions in string/M-theory theory,
following our discussion in \cite{SS19b} to which we refer for
full details and further pointers.

\medskip

\noindent {\bf Chord diagram observables from Hypothesis H.}
While informal considerations of quantum physics of branes in string theory
has proven to be a rich source for mathematical insights in quantum topology,
the underlying mathematical formulation of non-perturbative brane physics itself
(``M-theory'') has remained wide open (see \cite[p. 3 \& 6]{SS19b} for pointers).
Recently we have explored the {\it Hypothesis H}
\cite[\S 2.5]{Sati13}\cite{FSS19b}\cite{FSS19c}\cite{SS19a}\cite{SS20M5BraneAnomalyCancellation}
\cite{SS20CharacterInTwistorialCohomotopy}\cite{SS21}
that the proper mathematical formulation of the {\it C-field}
-- which is the only field expected in M-theory, besides the
field of (super-)gravity -- is as a cocycle in
(twisted) {\it Cohomotopy theory}. We have shown (\cite[\S 2]{SS19b})
how such a hypothesis
implies that (the topological sector of) the {\it phase space }
of $N$ probe $D6 \!\!\perp\!\! D8$-brane intersections (in an ambient flat spacetime)
is homotopy-equivalent to the based loop space of the configuration space
of $N$ ordered points in $\mathbb{R}^3$:
\vspace{-1mm}
\begin{equation}
  \label{LoopSpaceOfConfigurationSpace}
  \Omega
  \Big(
    \underset{
      \mathclap{\scalebox{.5}{$\{1,\!\cdots\!\!N\}$}}
    }{
      \mathrm{Conf}
    }
    (\mathbb{R}^3)
  \Big)
  \,.
\end{equation}

\vspace{-1mm}
\noindent
This implies (\cite[\S 2.5]{SS19b})
that the higher homotopical {\it observables} on such brane systems, conceptualized
as the homology of the phase space, is (by \cite[Thm. 4.1]{Kohno02})
nothing but the algebra of horizontal chord diagrams $\mathcal{A}^{{}^{\mathrm{pb}}}_{N}$
from Def. \ref{AlgebraOfAndWeightSystemsOnHorizontalChordDiagrams},
as shown in the following diagram:

\vspace{-4.5mm}
\begin{equation}
\label{FromCohomotopyToChordDiagramObservables}
\hspace{-5mm}
  \hspace{-3cm}
  \begin{tikzcd}[row sep=1pt, column sep=23pt]
    \mathclap{
    \mbox{
      \tiny
      \color{darkblue}
      \bf
      \begin{tabular}{c}
        weight
        \\
        systems
      \end{tabular}
    }
    }
    &
    \mathclap{
    \mbox{
      \tiny
      \color{darkblue}
      \bf
      \begin{tabular}{c}
        horizontal
        \\
        chord diagrams
      \end{tabular}
    }
    }
    &
    \mathclap{
    \mbox{
      \tiny
      \color{darkblue}
      \bf
      \begin{tabular}{c}
        homology of
        \\
        loop space of
        \\
        configuration space of
        \\
        ordered points in $\mathbb{R}^3$
      \end{tabular}
    }
    }
    &
    \mathclap{
    \mbox{
      \tiny
      \color{darkblue}
      \bf
      \begin{tabular}{c}
        loop space of
        \\
        configuration space of
        \\
        ordered points in $\mathbb{R}^3$
      \end{tabular}
    }
    }
    &
    \mathclap{
    \mbox{
      \tiny
      \color{darkblue}
      \bf
      \begin{tabular}{c}
        loop space of
        \\
        diff. Cohomotopy cocycle space of
        \\
        spacetime transversal to
        \\
        intersecting codim 3/codim 1-branes
      \end{tabular}
    }
    }
    \\
    \underset{\mathclap{\scalebox{.6}{$N \!\in\! \mathbb{N}$}}}{\bigoplus}
    \;
    \big(
      \mathcal{A}^{{}^{\mathrm{pb}}}_N
    \big)^\ast
    \ar[
      r,
      <->
    ]
    &
    \underset{\mathclap{\scalebox{.6}{$N \!\in\! \mathbb{N}$}}}{\bigoplus}
    \;
    \mathcal{A}^{{}^{\mathrm{pb}}}_N
    \ar[
      r,
      phantom,
      "\simeq"{description}
    ]
    \!\!\!\!\!\!
    &
    \!\!\!\!\!\!
    H_\bullet
    \left(
      \;
      \underset{\mathclap{\scalebox{.6}{$N \!\!\in\!\! \mathbb{N}$}}}{\sqcup}
      \;
      \Omega \underset{\mathclap{\scalebox{.5}{$\{1,\!\cdots\!\!, N\}$}}}{\mathrm{Conf}}(\mathbb{R}^3)
    \right)
    &
    \;\;
    \underset{\mathclap{\scalebox{.6}{$N \in \mathbb{N}$}}}{\sqcup}
    \;
    \Omega \underset{\mathclap{\scalebox{.5}{$\{1,\!\cdots\!\!, N\}$}}}{\mathrm{Conf}}(\mathbb{R}^3)
    \ar[l,|->]
    \ar[
      r,
      phantom,
      "\simeq"{description}
    ]
    \!\!\!\!\!\!
    &
    \!\!\!\!\!\!
    \Omega
    \,
    \boldpi^4_{\mathrm{diff}}
    \left(
      \begin{aligned}
        &
        (\mathbb{R}^3)^{\mathrm{cpt}}
          \wedge
        (\mathbb{R}^1)_{+}
        \\
        \cup
        \!\!\!\!\!\!
        &
        (\mathbb{R}^3)_{+\;}
         \! \wedge
        (\mathbb{R}^1)^{\mathrm{cpt}}
      \end{aligned}
    \right)
    \\
    &
    {}
    \ar[
      r,
      phantom,
      "{
        \mbox{
          \tiny
          \color{purple}
          \bf
          quantum observables
        }
      }"
    ]
    &
    {}
    &
    {}
    \ar[
      r,
      phantom,
      "{
        \mbox{
          \tiny
          \color{purple}
          \bf
          phase space (topological sector)
        }
      }"{near start}
    ]
    &
    {}
    \\
    \mathllap{
        \Big(
    \;
    }
    \underset{\mathclap{\scalebox{.6}{$N \!\in\! \mathbb{N}$}}}{\bigoplus}
    \;
    \big(
      \mathcal{A}^{{}^{\mathrm{pb}}}_N
    \big)^\ast
    \mathrlap{
    \Big)_{
        \mathrm{normalized}
        \atop
        {\& \; \mathrm{positive}}
    }}
    \ar[uu,hook]
    \\
    \vspace{-1cm}
    \mathclap{
    \mbox{
      \tiny
      \color{purple}
      \bf
 %     \begin{tabular}{c}
        quantum states
 %     \end{tabular}
    }
    }
    \\
    {\phantom{A}}
    \\
    {\phantom{A}}
    \\
    {\phantom{A}}
  \end{tikzcd}
  \hspace{-3cm}
\end{equation}
\vspace{-2cm}

\noindent
(Here $H_\bullet(-)$ is ordinary homology with complex coefficients,
$\boldpi^4_{\mathrm{diff}}(-)$ is a presheaf of pointed mapping spaces into the 4-sphere
\cite[\S 2.3]{SS19b};
$(-)^{\mathrm{cpt}}$ is one-point compactification and $(-)_+$ is disjoint union with a base point.)

\medskip

\noindent {\bf Chord diagrams in stringy quantum physics.}
While it was well-known that chord diagrams organize the quantum observables
of perturbative Chern-Simons theory
(Vassiliev knot invariants,
\cite{BarNatan91}\cite{Kontsevich93}\cite{BarNatan95}\cite{AF96} \cite{BarNatanStoimenow96}),
we observed in
\cite[\S 4]{SS19b} that chord diagrams moreover govern several
more recent
proposals for aspects of intersecting brane physics, including:

\vspace{-.1cm}
\begin{itemize}

\vspace{-.2cm}
\item[{\bf (i)}]
the fuzzy/non-commutative geometry
of D-brane intersections seen via the non-abelian Dirac-Born-Infeld (DBI)
action functional
(\cite[\S 3.2]{RST04}, review in \cite[\S A]{MPRS06}\cite[\S 4]{McNamara06});

\vspace{-.2cm}
\item[{\bf (ii)}]
several quantum many-body models for brane/bulk holography:

\vspace{-.1cm}
\begin{itemize}

\vspace{-.2cm}
\item[{\bf (a)}]
dimer/bit-thread models for quantum error correction codes
(\cite{JGPE19}\cite{Yan20}, review in \cite[\S 4.2]{JahnEisert21});

\vspace{-.1cm}
\item[{\bf (b)}]
scattering amplitudes in bulk duals of the SYK model
(\cite{BerkoozNarayanSimon18}\cite{BINT18}, review in \cite{Narovlansky19}).

\end{itemize}

\end{itemize}
\vspace{-3mm}

\noindent This confluence of occurrences of chord diagrams in
quantum brane physics
(which seems to previously have gone unnoticed; e.g.
the authors of \cite{JGPE19}\cite{Yan20} refer to
chord diagrams as ``dimer'' or ``bit-thread'' networks)
finds, assuming Hypothesis H, a natural explanation and unification from the
result \eqref{FromCohomotopyToChordDiagramObservables}
that chord diagrams indeed constitute the
fundamental (topological) quantum observables on
intersecting quantum brane systems.

\medskip

\noindent
{\bf Weight systems as quantum states of branes.}
This allows us to proceed further and
next ask for a rigorous characterization, assuming Hypothesis H, of
possible
{\it quantum states of intersecting brane systems},
by asking for weight systems which are quantum states in the precise sense of
Def. \ref{NotionOfWeightSystemsThatAreQuantumStates}
--
and this is our Question \ref{TheQuestion}.

\medskip

\noindent {\bf Bound state of 2 M5-branes.}
In particular, we may now rigorously ask, assuming Hypothesis H,
whether two M5-branes may form a {\it bound state} of coincident branes
-- a statement
that is widely expected to be true and which is at the heart of some of the
deepest conjectures in contemporary string/M-theory,
but for which no actual theory existed.
\vspace{-4mm}
\begin{itemize}

\vspace{2mm}
\item[{\bf (i)}]
In \cite[\S 4.9]{SS19b} we explained how
the would-be bound state of $N^{(\mathrm{M5})}$ coincident
M5-branes (specifically: transversal M5-branes in a pp-wave background)
should correspond, under \eqref{FromCohomotopyToChordDiagramObservables}, to the
Lie algebra weight system $w_{(\mathfrak{gl}(2), \mathbf{N}^{(\mathrm{M5})} )}$
\eqref{LieAlgebraWeightSystems}
for the $N^{(\mathrm{M5})}$-dimensional irrep of
$\mathfrak{gl}(2)$,
whence the would-be bound state of 2 M5-branes corresponds to the
{\it fundamental} weight system  $w_{ (\mathfrak{gl}(2), \mathbf{2}) }$
from Def. \ref{FundamentalWeightSystemOfgl2}.

\vspace{-2mm}
\item[{\bf (ii)}]
That this be a {\it bound state} of M5-branes
-- as opposed to an unstable tachyonic ``ghost'' state --
means to ask whether it is positive as a linear functional on observables
(Def. \ref{States})
and hence whether it is
a quantum state in the precise sense of Def. \ref{NotionOfWeightSystemsThatAreQuantumStates}.
That this is the case is the result of our
Thm. \ref{FundamentalglNWeightSystemsAreQuantumStates}!
-- which we thus may think of as a {\it no-ghost theorem} for
bound M5-brane states.

This establishes the result announced in \cite[\S 3.5]{SS19b}.
\end{itemize}

\noindent {\bf Bound state of $N^{(\mathrm{M5})}$ M5-branes.}
The natural next question to ask is whether $N^{(\mathrm{M5})}$ coincident
M5-branes form bound states, in this same sense, for any $N^{(\mathrm{M5})} \geq 2$,
hence whether the non-fundamental
Lie algebra weight systems $w_{(\mathfrak{gl}(2), \mathbf{N}^{(\mathrm{M5})} )}$
are quantum states on chord diagrams for $N^{(\mathrm{M5})} > 2$.
We hope to discuss this question elsewhere.

\medskip

\vspace{.5cm}
\noindent David Corfield, {\it Centre for Reasoning, University of Kent, UK.}
\\
 {\tt d.corfield@kent.ac.uk}
\\
\\
\noindent  Hisham Sati, {\it Mathematics, Division of Science, New York University Abu Dhabi, UAE.}
\\
{\tt hsati@nyu.edu}
\\
\\
\noindent  Urs Schreiber, {\it Mathematics, Division of Science, New York University Abu Dhabi, UAE; on leave from Czech Academy of Science, Prague.}
\\
{\tt us13@nyu.edu}

\end{document}